\documentclass[12pt, reqno]{amsart}
\usepackage{amsmath}
\usepackage{amssymb}
\usepackage{epsfig}
\usepackage{graphicx}
\usepackage{color}
\definecolor{shadecolor}{gray}{0.875}
\usepackage{amscd}
\usepackage{comment}
\usepackage{enumitem}
\usepackage{mathrsfs}

\usepackage[colorlinks=true]{hyperref}

\usepackage[OT2,T1]{fontenc}
\DeclareSymbolFont{cyrletters}{OT2}{wncyr}{m}{n}
\DeclareMathSymbol{\Sha}{\mathalpha}{cyrletters}{"58}

\numberwithin{equation}{section}

\input xy
\xyoption{all}

\calclayout
\allowdisplaybreaks[3]

\theoremstyle{plain}
\newtheorem{prop}{Proposition}[section]

\newtheorem{theo}[prop]{Theorem}
\newtheorem{coro}[prop]{Corollary}

\newtheorem{lemm}[prop]{Lemma}

\theoremstyle{definition}
\newtheorem{defi}[prop]{Definition}

\newtheorem{conj}[prop]{Conjecture}

\newtheorem{rema}[prop]{Remark}

\newtheorem{prin}[prop]{Principle}
\newtheorem{exam}[prop]{Example}

\def\rk{{\mathrm{rk}}}

\def\rH{{\mathrm H}}

\def\Br{\mathrm{Br}}

\def\Pic{\mathrm{Pic}}

\def\Nef{\mathrm{Nef}}

\def\Supp{\mathrm{Supp}}

\def\Spec{\mathrm{Spec}}

\def\Pic{\mathrm{Pic}}

\makeatother
\makeatletter

 \author{Qile Chen}
\address{Department of Mathematics \\
Boston College  \\
Chestnut Hill, MA \, \, 02467}
\email{qile.chen@bc.edu}

\author{Brian Lehmann}
\address{Department of Mathematics \\
Boston College  \\
Chestnut Hill, MA \, \, 02467}
\email{lehmannb@bc.edu}

\author{Sho Tanimoto}
\address{Graduate School of Mathematics, Nagoya University, Furocho Chikusa-ku, Nagoya, 464-8602, Japan}
\email{sho.tanimoto@math.nagoya-u.ac.jp}

\title[Manin's conjecture for semi-integral curves]{Manin's conjecture for semi-integral curves and $\mathbb A^1$-connectedness}

\begin{document}
\date{\today}

\begin{abstract} 
We explore log Manin's conjecture for integral points and its connections to $\mathbb A^1$-connectedness. We prove log Manin's conjecture for Campana rational curves and for $\mathbb A^1$-curves on split toric varieties.  Our arguments combine the Cox ring description of the moduli space of rational curves with Batyrev's heuristic-type counting arguments.  As our proofs are geometric in nature, they give a geometric explanation of the mysterious leading constant for Campana points proposed in \cite{CLTBT}.
\end{abstract}

\maketitle

\setcounter{tocdepth}{1}
\tableofcontents

\section{Introduction}

One of the central themes in diophantine geometry is the distribution of rational points on algebraic varieties, and one of the main questions here is the asymptotic formula for the counting function of rational points of bounded height on a projective variety defined over a global field. This leads to Manin's conjecture predicting the asymptotic formula for the counting function of rational points on a smooth Fano variety defined over a number field (\cite{FMT, BM, Peyre, BT, Peyre03, Peyre17, LST18, LS24}), and its geometric counterparts over function fields of curves (\cite{Bat88, Bou09, Bou, Bou12, Bourqui13, Peyre12, LT17, LRT23, LT26, CLL16, Bilu23}). See \cite{Tanimoto26} and \cite{LRT23, LT26} for modern formulations of Manin's conjecture.

Around a decade ago, Chambert-Loir, Takloo-Bighash, and Tschinkel produced a series of works on log Manin's conjecture for integral points (\cite{CLT12, TBT13, CLTtoric}). Here integral points are associated to a log pair $(\underline{X}, \Delta)$ such that $\Delta$ is a reduced effective divisor. However, the proof of log Manin's conjecture for integral points on toric varieties in \cite{CLTtoric} had a non-trivial gap. This situation sparked extensive activities in the field, and it led to the study of log Manin's conjecture for semi-integral points including Campana points, Darmon points, and $\mathcal M$-points (\cite{BY21, PSTVA, Streeter22, SS24, Shute22, CLTBT, Moerman1, Moerman2}). Semi-integral points are certain rational points associated to a log pair $(\underline{X}, \Delta)$ with $\Delta$ being an effective $\mathbb Q$-divisor with standard coefficients, and they interpolate between the notion of rational points and integral points. Moreover log Manin's conjecture for integral points has also witnessed major breakthroughs in \cite{Wilsch} and \cite{Santens}; the latter paper finished the proof of log Manin's conjecture for integral points on toric varieties.  \cite{Santens} also proposed a general formulation of log Manin's conjecture for integral points.

In this paper we have two goals.  First, we make a new observation concerning the geometric aspects of the formulation of log Manin's conjecture for integral points.  Second, we prove log Manin's conjecture for Campana rational curves and $\mathbb A^1$-curves on split smooth projective toric varieties defined over finite fields.  

Over global function fields, the most closely related prior work is \cite{Faisant25campana} which analyzes Campana points on toric varieties using the Grothendieck ring of varieties; see Section \ref{sect:history} for more background.

\subsection{Main results}

We explain the main results in this paper:

\subsubsection*{$\mathbb A^1$-connectedness}

Let $(\underline{X}, \Delta)$ be a pair such that $\underline{X}$ is a smooth projective variety defined over a number field $F$ and $\Delta$ be an effective reduced divisor on $\underline{X}$. We fix a projective integral model $\mathcal X$ of $\underline{X}$ over $\Spec \, \mathfrak o_F$ where $\mathfrak o_F$ is the ring of integers for $F$.
We denote the flat closure of $\Delta$ in $\mathcal X$ by $\widetilde{\Delta}$ and let $\mathcal U = \mathcal X \setminus \mathrm{Supp}(\widetilde{\Delta})$. Then the set of $\mathfrak o_F$-integral points is $\mathcal U(\mathfrak o_F)$. \cite{Santens} resolved log Manin's conjecture for integral points on a smooth projective toric variety and proposed a formulation of log Manin's conjecture for integral points when $(\underline{X}, \Delta)$ is a log Fano variety. Santens used the following assumption throughout his paper:
\[
\overline{F}[\underline{U}]^\times = \overline{F}^\times,
\]
where $\underline{U}$ is the generic fiber of $\mathcal U$.

In this paper, we make the following observation:
\begin{prop}[Proposition~\ref{prop:equivconditionstoric}]
    Let $\mathbf k$ be an algebraically closed field and $\underline{X}$ be a smooth projective toric variety defined over $\mathbf k$ with an SNC reduced boundary divisor $D$ consisting of some torus invariant components 
    with $\underline{U} = \underline{X} \setminus \mathrm{Supp}(D)$. Then the following statements are equivalent:
    \begin{itemize}
        \item $\mathbf k [\underline{U}]^\times = \mathbf k^\times$, and;
        \item $\underline{U}$ is $\mathbb A^1$-connected.
    \end{itemize}
\end{prop}

Note that in general if $\underline{U}$ is $\mathbb A^1$-connected, then $\mathbf k [\underline{U}]^\times = \mathbf k^\times$ is true (see Lemma~\ref{lemm:A1connectedimpliesnoinvertible}). However the converse is not true in general:
\begin{theo}[Example~\ref{rema:counterexample}]
    Let $\mathbf k$ be an algebraically closed field of characteristic $0$. Then there exists a weak log Fano variety $(\underline{X}, \Delta)$ with an SNC  boundary $\Delta$ such that $\underline{U} = \underline{X} \setminus \mathrm{Supp}(\Delta)$ satisfies
    \[
    \mathbf k [\underline{U}]^\times = \mathbf k^\times,
    \]
    but $\underline{U}$ is not $\mathbb A^1$-connected. Moreover there is a surjective morphism $$\phi : (\underline{X}, \Delta) \to (\underline{Y}, \Delta_Y),$$ such that $\mathrm{Supp}(\Delta) \supset \mathrm{Supp}(\phi^{-1}(\Delta_Y))$ and $(\underline{Y}, \Delta_Y)$ is a log Calabi-Yau pair, i.e., $K_{\underline{Y}} + \Delta_Y \sim 0$.
\end{theo}

In particular, integral points on $(\underline{X}, \Delta)$ map to integral points on $(\underline{Y}, \Delta_Y)$ and it seems unlikely that log Manin's conjecture should hold in this example.
\cite[Conjecture 6.1]{Santens} proposed a version of log Manin's conjecture for integral points.  In this conjecture Santens assumes that $\overline{F}[\underline{U}]^\times = \overline{F}^\times$.  In view of the above discussion, we propose the following modification:

\begin{conj}
    \cite[Conjecture 6.1]{Santens} holds under the additional assumption that $\underline{U}$ is geometrically $\mathbb A^1$-connected.
\end{conj}

\begin{rema} 
    Analytic obstructions as in \cite{Wilsch} and \cite{Santens} play an important role in counting integral points on toric varieties.  We show that they also admit a natural interpretation via the geometry of $\mathbb{A}^{1}$-curves; see Proposition \ref{prop:obstruction}.
\end{rema}

\subsubsection*{Campana rational curves on split toric varieties}

Let $\mathbf k = \mathbb F_q$ be a finite field and $F = \mathbb F_q(t)$. Let $\underline{X}$ be a smooth projective toric variety defined over $\mathbf k$ of dimension $n$ with the open orbit $\underline{T}$ such that the boundary divisor $\Delta = \sum_i \Delta_i$ is an SNC divisor. We also assume that $\underline{X}$ is split.
For each $i$, we associate to $\Delta_{i}$ a positive integer $m_i \in \mathbb Z_{>0}$.  We let $(X, \Delta_{\bold m})$ be the Campana orbifold with
\[
\Delta_{\bold m} = \sum_i \left( 1- \frac{1}{m_i}\right)\Delta_i.
\]
\begin{defi}
A rational curve $f : \underline{\mathbb P}^1 \to \underline{X}$ with $\underline{T}\cap f(\underline{\mathbb P}^1) \neq \emptyset$ is a Campana curve if for every $i$, each point in the support of $f^*\Delta_i$ occurs with multiplicity $\geq m_i$.
\end{defi}
Log Manin's conjecture for Campana rational curves concerns the asymptotic formula for the counting function of Campana rational curves of bounded degree $\deg (-f^*(K_{\underline{X}} + \Delta_{\bold m}))$.
Regarding this conjecture we prove
\begin{theo}[Theorem~\ref{theo:Campana}]
\label{theo:Campanaintro}
    An analogue of log Manin's conjecture for Campana points on $(X, \Delta_{\bold m})$ holds for Campana rational curves with respect to the log anticanonical class as predicted by \cite{PSTVA}. Moreover, the leading constant is compatible with the conjecture in \cite{CLTBT}. 
\end{theo}
See Section~\ref{sec:Campana} for a more precise statement.
In particular, the leading constant admits a description involving a finite sum of height integrals twisted by the unramified automorphic characters induced by Brauer elements as predicted by \cite{CLTBT}. See Section~\ref{subsec:leadingconstandforCampana} for more details.

We should note that \cite[p. 7]{Faisant25campana} claims that his motivic proof for a motivic version of log Manin's conjecture for Campana curves on split toric varieties can be applied directly to the point counting argument over $\mathbb F_q$ (without a detailed explanation). However, his argument does not seem to involve the automorphic characters which we observe here.

\subsubsection*{$\mathbb A^1$-curves on split toric varieties}

We continue to work in the setting of a split smooth projective toric variety $\underline{X}$ defined over a finite field $\mathbf k = \mathbb F_q$. We fix a reduced boundary divisor
\[
D = \sum_{i \in \mathcal A} \Delta_i \leq \Delta.
\]
Let $X$ be the log scheme associated to $(\underline{X}, D)$.

\begin{defi}
A rational curve $f : \underline{\mathbb P}^1 \to \underline{X}$ with $\underline{T}\cap f(\underline{\mathbb P}^1) \neq \emptyset$ is an $\mathbb A^1$-curve if for any $i \in \mathcal A$, the support of $f^*\Delta_i$ is contained in $\{\infty\}$.
\end{defi}

Log Manin's conjecture for $\mathbb A^1$-curves concerns the asymptotic formula for the counting function of $\mathbb A^1$-curves of bounded degree $$\deg (-f^*K_X) = \deg (-f^*(K_{\underline{X}} + D)).$$
Regarding this conjecture we prove

\begin{theo}[Theorem~\ref{theo:A1curve}]
    \label{theo:A1curvesintro}
    Suppose that $(X, D)$ is geometrically $\mathbb A^1$-connected.
    Then an analogue of log Manin's conjecture for integral points on $(X, D)$ holds for $\mathbb A^1$-curves with respect to the log anticanonical class in the sense of \cite{Santens}.
\end{theo}

See Section~\ref{sec:A1} for a detailed statement.

\subsubsection*{The method of proof}

Our methods for the proofs of both Theorem~\ref{theo:Campanaintro} and \ref{theo:A1curvesintro} are based on the birational geometry of the moduli space of rational curves on a smooth projective toric variety combined with Batyrev's heuristic-type counting arguments proposed originally by Batyrev in \cite{Bat88} and subsequently developed in \cite{LRT23, DLTT25, LT26}. In particular, it relies on an extensive usage of the virtual height zeta function, i.e., a certain generating series of the number of $\mathbf k$-points on the moduli space, studied in \cite{DLTT25} for the space of rational curves on quartic del Pezzo surfaces, and its realization as certain height integrals developed in \cite{CLT10}. 

The birational geometry of the moduli space of rational curves on a toric variety is inspired by the study of the space of rational curves on del Pezzo surfaces in \cite{DLTT25}, and it is also viewed as applications of the Cox ring and the universal torsor description explored in \cite{Sal98, Bou, Bou16, Faisant25campana}. In particular, we do not invoke any harmonic analysis on tori or the Poisson summation formula.

One remarkable feature of our proofs is a Batyrev's heuristic-type geometric explanation of the leading constant for Campana points, proposed in \cite{CLTBT}, expressing the constant as a sum of certain height integrals twisted by automorphic characters induced by Campana Brauer elements. This can be viewed as a convincing evidence towards the constants proposed by \cite{CLTBT}.

\subsection{History} \label{sect:history}

\subsubsection*{Manin's conjecture for toric varieties}

Let us briefly explain the history of Manin's conjecture for toric varieties.

Manin's conjecture for toric varieties over number fields has been proved by Victor Batyrev and Yuri Tschinkel in \cite{BT95, BT96, BT98, CLT01}. Over global function fields, this has been studied by Bourqui in \cite{Bou03, Bou11} and its motivic analogue has been developed in \cite{Bou09, BDH22}. These papers are based on harmonic analysis on tori. A different approach using universal torsors was initiated by Salberger in \cite{Sal98} and further explored in \cite{Pieropan16}.

Log Manin's conjecture for integral points on toric varieties has been pioneered in \cite{CLTtoric} and \cite{Wilsch}, and it has been finally proved in \cite{Santens} by applying harmonic analysis to universal torsors. Log Manin's conjecture for Campana/Darmon points has been pioneered in \cite{Streeter22, PS24} and finally proved in \cite{SS24} using harmonic analysis. Its generalizations to further semi-integral points such as $\mathcal M$-points have been explored in \cite{Moerman2} using the universal torsor approach. A motivic analogue for Campana curves on toric varieties has been established in \cite{Faisant25} and \cite{Faisant25campana} using the universal torsor approach. Our approach is indirectly related to the methods used by Faisant, but our method is more direct.  In particular, we do not invoke inclusion-exclusion using the M\"obius function.  A key advantage of our approach is that it allows us to identify a ``virtual'' height zeta function with the height integrals studied in \cite{CLT10} for split toric varieties.

\subsubsection*{Campana curves and $\mathbb A^1$-curves}

The geometry of $\mathbb A^1$-curves has been pioneered in \cite{KM99}, and further investigated by the first author and Yi Zhu, e.g., \cite{CZ15, CZ17, CZ18, CC21} using the moduli stack of stable log maps. The geometry of Campana rational curves has been first explored by Campana himself in \cite{Campana04, Campana07, Campana10, Campana11}.  Recently \cite{CLT25} set the foundation for the study of Campana curves using the geometry of stable log maps. This has been further investigated in \cite{Enhao}.

\bigskip

\noindent
{\bf Acknowledgements:}
The authors would like to thank Dan Loughran for his help regarding Campana Brauer groups. The authors would also like to thank Takehiko Yasuda for answering our questions regarding toric singularities.
The authors would like to thank Dan Loughran, Tim Santens, and Sam Streeter for comments on a draft of this paper.
The authors also would like to thank Ramin Takloo-Bighash for pointing out a mistake in the original statement of Theorem~\ref{theo:Campana}, which leads us to consider averaging procedure. Finally we would like to thank the anonymous referee for a careful reading of the paper and many useful suggestions.

Qile Chen is supported in part by Simons Travel Support for Mathematicians.
Brian Lehmann was supported by Simons Foundation grant Award Number 851129.  Sho Tanimoto was partially supported by JST FOREST program Grant number JPMJFR212Z and by JSPS KAKENHI Grant-in-Aid (B) 23K25764.

ChatGPT-Astra and Claude were used to check the correctness of the paper during the revision.

\section{Preliminaries}

\noindent
{\bf Notation}: 
We freely use the notation established in \cite{CLT25}. In particular, 
usual schemes are denoted by $\underline{X}, \underline{Y}$, and so on, and log schemes are denoted by $X, Y$, and so on. The ground field is denoted by $\mathbf k$, and a variety defined over $\mathbf k$ is a separated integral scheme of finite type over $\mathbf k$. We denote the log scheme associated to a pair $(\underline{X}, \Delta)$ by $(X, \Delta)$ or $X$ if there is no confusion. For a $\mathbf k$-variety $\underline{X}$ and an extension $\mathbf k'/\mathbf k$, we denote its base change by $\underline{X}_{\mathbf k'}$. We also fix the separable closure $\mathbf k^s$ of $\mathbf k$.

Let $\underline{X}$ be a projective variety over $\mathbf k$. We denote the space of numerical real $1$-cycles by $N_1(\underline{X})$ and denote the lattice generated by integral classes by $N_1(\underline{X})_{\mathbb Z} \subset N_1(\underline{X})$. We denote the nef cone of curves by $\mathrm{Nef}_1(\underline{X}) \subset N_1(\underline{X})$ and denote its interior by $\mathrm{Nef}_1(\underline{X})^{\circ}$.
For any cone $\mathsf C \subset N_1(\underline{X})$, we denote $\mathsf C \cap N_1(\underline{X})_{\mathbb Z}$ by $\mathsf C_{\mathbb Z}$.
We denote the dimension of $N_1(\underline{X})$ by $\rho(\underline{X})$.

\subsection{Preliminaries on toric varieties}

\subsubsection*{Toric pairs}

Let $\mathbf k$ be a field.
A $\mathbf k$-torus is an algebraic group $\underline{T}$ defined over $\mathbf k$ such that $\underline{T}_{\mathbf k^s}$ is isomorphic to $\underline{\mathbb G}_m^n$. A toric variety is a variety $\underline{X}$ with an action by a $\mathbf k$-torus $\underline{T}$ with the unique open orbit isomorphic to $\underline{T}$. 
  We let $\Delta = \sum_i \Delta_i$ denote the full reduced toric boundary divisor for $\underline{X}$. When $\underline{X}$ is split, i.e., every boundary component is geometrically irreducible and $\underline{T}$ is isomorphic to $\underline{\mathbb G}_m^n$, we denote the fan associated to $\underline{X}$ by $\Sigma$. The set of $1$-dimensional cones is denoted by $\Sigma^{(1)}$ which corresponds to the set of boundary components $\{\Delta_i\}$. The ray generator corresponding to $\Delta_i$ is denoted by $v_i$.
We also have the N\'eron-Severi torus $\underline{T}_{\mathrm{NS}} = \mathrm{Hom}(\mathrm{NS}(\underline{X}), \underline{\mathbb G}_m)$ which is a subtorus of $\underline{\mathbb G}_m^{\Sigma^{(1)}}$ and we have a natural identification
\[
\underline{\mathbb G}_m^{\Sigma^{(1)}}/\underline{T}_{\mathrm{NS}} \cong \underline{T}.
\]

First we recall the following definition:

\begin{defi}
A pair $(\underline{X},D)$ is a toric pair if $\underline{X}$ is a toric variety defined over $\mathbf k$ and $D \subset \underline{X}$ denotes a reduced sum of $\underline{T}$-invariant divisors. Note that $D$ may not be the full boundary divisor. The log scheme associated to $(\underline{X}, D)$ is denoted by $(X, D)$ or by $X$ when the boundary is understood in context.

We say that a toric pair is smooth if $\underline{X}$ is smooth and the full boundary $\Delta$ has SNC support.  
\end{defi}

\subsubsection*{Heights and the Tamagawa measures}

We fix notation from number theory:

\noindent
{\bf Notation}: Let $\mathbf k = \mathbb F_q$ be a finite field and let $\underline{C}$ be a smooth geometrically integral projective curve defined over $\mathbf k$. Let $F = \mathbf k(\underline{C})$ be the global function field of $\underline{C}$. We denote the set of closed points on $\underline{C}$ by $|\underline{C}|$. For each $c \in |\underline{C}|$, we denote the completion of $F$ with respect to the discrete valuation $v_c$ associated to $c$ by $F_c$.
Let $\mathfrak o_c$ be the ring of integers for $F_c$. Let $\mathbf k_c$ be the residue field of $\mathfrak o_c$ and $q_c$ be the size of $\mathbf k_c$. We also denote $[\mathbf k_c: \mathbf k]$ by $|c|$ so that we have $q_c = q^{|c|}$.
For any $f \in F_c$, we define the norm $|\cdot|_c$ by
\[
|f|_c = q_c^{-v_c(f)}.
\]
This satisfies the product formula: for any $f \in F$,
\begin{equation}
\label{equation:productformula}
    \prod_{c \in |\underline{C}|}|f|_c = 1.
\end{equation}
We also denote the adelic ring of $F$ by $\mathbb A_F$.

Finally let us define the Hasse-Weil zeta function for $\underline{\mathbb P}^1$:
    \[
    \zeta_{\underline{\mathbb P}^1}(t) = \prod_{c \in |\underline{\mathbb P}^1|}(1-q_c^{-t})^{-1} = \frac{1}{(1-q^{-t})(1-q^{-(t-1)})}.
    \]

Let $\underline{X}$ be a smooth projective toric variety defined over $\mathbf k$ of dimension $n$ with the open orbit $\underline{T}$ such that the full boundary divisor $\Delta = \sum_i \Delta_i$ is an SNC divisor. For simplicity we also assume that $\underline{X}$ is split.
We recall the height theory of $\underline{X}$ following the exposition of \cite{CLT10}.
First we recall the definition of adelic metrics induced by the trivial family $\underline{X}\times \underline{\mathbb P}^1$:

\begin{defi}[{\cite[Section 2.1.5]{CLT10}}]
    For each $c \in |\underline{\mathbb P}^1|$, the metric $\| \cdot \|_c$ induced on $\mathcal L_i = \mathcal O(\Delta_i)$ is defined by the following property: for any $x \in \underline{X}(F_c)$ and $\ell_x \in \mathcal L_{i, x}(F_c)$, we have
    \[
    \|\ell_x\|_c \leq 1 \iff \ell_x \in \mathcal L_{i, x}(\mathfrak o_c).
    \]
    We call the collection $(\|\cdot \|_c)_{c \in |\underline{\mathbb P}^1|}$ the adelic metric.
\end{defi}

Using metrics, one can define height functions:
\begin{defi}[{\cite[Section 2.3]{CLT10}}]
Let $\Sigma^{(1)}$ be the set of boundary divisors.  For each $i \in \Sigma^{(1)}$ let $\mathsf s_i \in H^0(\underline{X}, \mathcal O(\Delta_i))$ be a section corresponding to $\Delta_i$. We assign a complex number $t_{i}$ to each $i \in \Sigma^{(1)}$ and set $\bold t = (t_i)$.

For $c \in |\underline{\mathbb P}^1|$, we define the local height function $$\mathsf H_c : \mathbb C^{\Sigma^{(1)}} \times \underline{T}(F_c) \to \mathbb C^\times,$$ by
\[
\mathsf H_c(\bold t, g_c) = \prod_i \|\mathsf s_i(g_c)\|_c^{-t_i}.
\]
We define the global height function $$\mathsf H : \mathbb C^{\Sigma^{(1)}} \times \underline{T}(\mathbb A_F) \to \mathbb C^\times,$$ by
\[
\mathsf H(\bold t, (g_c)) = \prod_{c \in |\underline{\mathbb P}^1|} \mathsf H_c(\bold t, g_c).
\]
When $L = \sum_{i \in \Sigma^{(1)}} \lambda_i \Delta_i$, 
we write $\mathsf H((\lambda_i), (g_c))$ as $\mathsf H(L, (g_c))$ and so on.
\end{defi}

Next we define the local Tamagawa measures:

\begin{defi}[{\cite[Section 2.1.10]{CLT10}}]
    Let $\omega$ be a non-vanishing top degree $\underline{T}$-invariant form on $\underline{T}$, which is unique up to scaling.
    Then we consider the corresponding divisor
    \[
    -\mathrm{div}(\omega) = \sum_i \Delta_i.
    \]
    For $c \in |\underline{\mathbb P}^1|$, we define the norm function $\|\omega\|_c$ as
    \[
    \|\omega\|_c = \prod_i \|\mathsf s_i\|^{-1}_c.
    \]
    The top form $\omega$ defines a Haar measure $|\omega |$ on $\underline{T}(F_c)$ 
    and we define the local Tamagawa measure on $\underline{X}(F_c)$ by
    \[
    \tau_c = \frac{|\omega |}{\|\omega\|_c}.
    \]
    \end{defi}
    Let $D \leq \Delta$ be a reduced effective sum of $T$-invariant divisors and set $\underline{X}^\circ = \underline{X} \setminus \mathrm{Supp}(D)$.
    Let $\mathrm{EP}(\underline{X}^\circ)$ be the following virtual $\mathbb Q$-Galois module:
    \[
    [H^0(\underline{X}^\circ_{F^{s}}, \mathbb G_m)/(F^{s})^\times]_{\mathbb Q} - [H^1(\underline{X}^\circ_{F^{s}}, \mathbb G_m)]_{\mathbb Q}, 
    \]
    as in \cite[Definition 2.2]{CLT10}.
    Using this, one can define the Artin $L$-function as in \cite[After Definition 2.2]{CLT10}:
    \[
    L(t, \mathrm{EP}(\underline{X}^\circ)) = \prod_{c \in |\underline{\mathbb P}^1|} L_c(t, \mathrm{EP}(\underline{X}^\circ)).
    \]
    Here \[L_c(t, \mathrm{EP}(\underline{X}^\circ)) =\det(1-q_{c}^{-t} \mathrm{Fr}_{c} \, | \, \mathrm{EP}(\underline{X}^{\circ})^{\Gamma^{0}_{c}})^{-1} \] where $\mathrm{Fr}_{c}$ is the geometric Frobenius and $\Gamma^{0}_{c}$ is an inertia subgroup at $c$.
    The function $L(t, \mathrm{EP}(\underline{X}^\circ))$ is $(2\pi\sqrt{-1}/\log q)$-periodic and admits a holomorphic continuation to the half plane $\Re(t) \geq 1 - \delta$ for some $\delta > 0$ with possible poles at $t = 1 + k\frac{2\pi\sqrt{-1}}{\log q}$ for $k \in \mathbb Z$. Let $b = \mathrm{ord}_{t = 1}L(t, \mathrm{EP}(\underline{X}^\circ))$ and define
    \[
    L_*(1, \mathrm{EP}(\underline{X}^\circ)) = \lim_{t \to 1}(1-q^{-(t-1)})^{-b}L(t, \mathrm{EP}(\underline{X}^\circ)),
    \]
    as in \cite[p. 368]{CLT10}.
    With these definitions, we introduce the following:
    \begin{defi}[{\cite[Definition 2.8]{CLT10}}]
        The Tamagawa measure $\tau_{X^\circ}$ on $X^\circ(\mathbb A_F)$
        is 
        \[
        \tau_{\underline{X}^\circ} := L_*(1, \mathrm{EP}(\underline{X}^\circ))^{-1}\prod_{c\in |\underline{\mathbb P}^1|} L_c(1, \mathrm{EP}(\underline{X}^\circ))\tau_c.
        \]
    \end{defi}

\subsection{$\mathbb A^1$-connectedness} \label{sss:A1-curve}

\subsubsection*{$\mathbb A^1$-curves}
Denote by $\underline{\mathbb{A}}^1$ and $\underline{\mathbb{P}}^1$ the  affine line and projective line over $\mathbf k$ respectively. 
An $\mathbb A^1$-curve on a $\mathbf k$-variety $\underline{U}$ is a non-constant proper morphism $\underline{f}$ to $\underline{U}$ from either $\underline{\mathbb A}^1$ or $\underline{\mathbb P}^1$.

Consider a log scheme $X$ given by the pair $ (\underline{X},\Delta)$ where $\underline{X}$ is a proper $\mathbf k$-variety, and $\Delta \subset \underline{X}$ is an SNC divisor defined over $\mathbf k$. 
A rational curve $\underline{f} \colon \underline{\mathbb{P}}^1 \to \underline X$ is called an $\mathbb A^1$-curve if $\underline{f}^*\Delta$ is either empty, or is supported at a single  $\mathbf k$-point of $\underline{\mathbb{P}}^1$, denoted by $\infty \in \underline{\mathbb{P}}^1$. 
In the latter case, let $\mathbb P^1$ be the log scheme given by the pair $(\underline{\mathbb{P}}^1, \infty)$. Then $\underline{f}$ lifts to a unique log map $f \colon \mathbb P^1 \to X$.

Denote by $\underline{U} = \underline{X}\setminus \Delta$. It is clear that each $\mathbb A^1$-curve on $X$ corresponds to a unique $\mathbb A^1$-curve of $\underline{U}$. 
Conversely, given $X$ as a compactification of $\underline{U}$, the properness of $\underline{X}$ implies that each $\mathbb A^1$-curve on $\underline{U}$ extends to a unique $\mathbb A^1$-curve on $X$. 
Note that when $\underline{f} \colon \underline{\mathbb{P}}^1 \to \underline{X}$ factors through $\underline{U}$, we view $\underline{f}$ as a unique $\mathbb A^1$-curve by fixing the choice of $\infty$. 

\subsubsection*{Free $\mathbb A^1$-curves}
Now assume that $\mathbf k$ is algebraically closed. We say that $\underline{U}$ is (separably) $\mathbb A^1$-uniruled if there is a family of $\mathbb A^1$-curves $\underline{f}_{\underline M} \colon \underline{\mathbb{A}}^1 \times \underline{M} \to \underline{U}$ over $\underline{M}$, such that $\underline{f}_{\underline M}$ is (separable) dominant. 
If furthermore, the two-evaluation morphism
\[
\underline{f}_{\underline M}\times_{\underline{M}} \underline{f}_{\underline M} \colon \underline{\mathbb{A}}^1 \times \underline{M} \times \underline{\mathbb{A}}^1 \to \underline{U}\times\underline{U}
\]
is (separable) dominant, we say that $\underline{U}$ is (separably) $\mathbb A^1$-connected. 

These two properties can be checked using a log compactification $X$ of $\underline{U}$ as follows. 
Denote by $T_X$ the log tangent bundle of $X$. 
An $\mathbb A^1$-curve $\underline{f} \colon \underline{\mathbb{P}}^1 \to \underline{X}$ is free (resp. very free) if $\underline{f}^*T_X$ is semi-ample (resp. ample). 
It is proven in \cite{CZ17} that $\underline{U}$ is separably $\mathbb A^1$-uniruled (resp. separably $\mathbb A^1$-connected) if and only if $X$ admits a free (resp. very free) $\mathbb A^1$-curve. 

Note that when $\underline{f}^*\Delta = \emptyset$, the above statements reduce to the equivalence between separable uniruledness (resp. separable rational connectedness) and existence of free (very free) rational curves on $\underline{X}$.

\subsubsection*{Contact orders}
We first assume that $\mathbf k$ is algebraically closed. 
Consider an $\mathbb A^1$-curve $f \colon \mathbb P^1 \to X$. Its contact order is defined to be its intersection numbers at $\infty$ with each irreducible component of the SNC divisor $\Delta$. More precisely, let $\Delta = \sum_i \Delta_i$ be the decomposition to irreducible components. 
Since $f^*\Delta$ is only allowed to be supported at $\infty$, the contact order of $f$ is given by the sequence of non-negative integers $(f_*[\mathbb P^1]\cap \Delta_i)_i$. 
Contact orders are invariant along deformations of $\mathbb A^1$-curves. 

Now consider a general (not necessarily closed) ground field $\mathbf k$. Assume that each irreducible component $\Delta_i$ is defined over $\mathbf k$. Given the deformation invariance of contact orders, we define the contact order of an $\mathbb A^1$-curve to be again the sequence of intersection numbers  $(f_*[\mathbb P^1]\cap \Delta_i)_i$.

\subsubsection*{Contact orders in the toric case}
In the toric case, the data of contact orders can be nicely encoded in the combinatorial structure of fans. 
Let $\underline{X}$ be a split toric variety with its unique open orbit given by a $\mathbf k$-torus $\underline{T}$, and with  its full toric boundary $\Delta = \sum_i \Delta_i$ as before.  
Denote by $\Sigma$ the fan of $\underline{X}$.
For the purpose of this paper, we assume that $\underline{X}$ is smooth. 
For a choice of torus-invariant divisor $D \subset \Delta$, let $X$ be the log scheme corresponding to the toric pair $(\underline{X}, D)$. 

We first consider the case that $D = \Delta$ is the full toric boundary. 
While $X$ does not admit any $\mathbb A^1$-curves, we may still define contact orders as local intersection numbers with respect to each irreducible component of $\Delta$. 
Then the set of contact orders can be identified with integral points in $|\Sigma|$. 
More precisely, let $c \in |\Sigma|$ be an integral point, and $\sigma \in \Sigma$ be the minimal cone containing $c$. Then the smoothness of $\sigma$ implies a unique decomposition  
\begin{equation}\label{eq:contact=integral-point}
c = \sum_{\rho \in \sigma^{(1)}} c_{\rho} v_{\rho},
\end{equation}
where $\sigma^{(1)}$ is the set of rays of $\sigma$, and $v_{\rho}$ is the ray generator of $\rho \in \sigma^{(1)}$. In this case $c_{\rho}$ is the prescribed local intersection number with the component of $\Delta$ corresponding to $\rho$. 

For a general $D \subset \Delta$, we define $\Sigma_X \subset \Sigma$ as the subset of cones spanned by rays given by irreducible components of $D$. Then the set of contact orders to $X$ is precisely integral points in $|\Sigma_X|$. 
Indeed, for any integral point $c \in |\Sigma_X|$, the minimal cone $\sigma \in \Sigma$ is contained in $\Sigma_X$ by construction. Thus, the integral point $c$ specifies the local intersection numbers with respect to each irreducible component of $D$ via the same decomposition \eqref{eq:contact=integral-point} as before. 

The fan $\Sigma_X$ is closely related to the Clemens complex of $X$ as in Definition \ref{defi:clemens-complex}.

\section{Integral points and $\mathbb{A}^{1}$-connectedness}

As discussed in the introduction, there have been several recent attempts to give a rigorous formulation of Manin's conjecture for integral points (\cite{Wilsch, Santens}).  In this section we discuss the geometry underlying these formulations, focusing on the following principle:

\begin{prin}
Let $(X,\Delta)$ be a log Fano pair over a global field.  The integral points should admit a good asymptotic description when $(X,\Delta)$ is $\mathbb{A}^{1}$-connected.
\end{prin}

\subsection{Toric pairs}

For a smooth toric pair, there are many equivalent ways of identifying the $\mathbb{A}^{1}$-connectedness condition.

\begin{prop} \label{prop:equivconditionstoric}
Assume that the ground field $\mathbf k$ is algebraically closed.
Let $(X_{\Sigma}, D)$ be a smooth projective toric pair associated to a fan $\Sigma \subset N_{\mathbb R}$.  Let $U = X_{\Sigma} \backslash \mathrm{Supp}(D)$.  
Then the following conditions are equivalent:
\begin{enumerate}
\item $(X_{\Sigma}, D)$ is $\mathbb{A}^{1}$-connected. 
\item $\mathbf k[U]^{\times} \cong \mathbf k^{\times}$.
\item The rays of $\Sigma$ corresponding to divisors not contained in $D$ span $N_{\mathbb{Q}}$.
\item The irreducible components $D_{i}$ of $D$ are linearly independent in $\Pic(X_{\Sigma})_{\mathbb{Q}}$.
\end{enumerate}
Also the following conditions are equivalent:
\begin{enumerate}[label=(\alph*)]
\item $(X_{\Sigma}, D)$ is separably $\mathbb{A}^{1}$-connected. 
\item The rays of $\Sigma$ corresponding to divisors not contained in $D$ span $N \otimes \mathbf{k}$.
\item The log tangent bundle $T_{X_{\Sigma}}$ does not admit any non-zero map to $\mathcal{O}_{X_{\Sigma}}$.
\end{enumerate}
\end{prop}

\begin{proof}
Note that $U$ is the toric variety defined by the fan $\Sigma_{D}$ which is obtained by removing every cone in $\Sigma$ which contains a ray corresponding to an irreducible component of $D$. In other words, $\Sigma_{D}$ consists of the cones in $\Sigma$ which are spanned by the rays corresponding to divisors not contained in $D$. 

(1) $\Rightarrow$ (2): Suppose that $\mathbf k[U]^{\times}$ has a non-constant function.  This induces a morphism $f: U \to \mathbb{G}_{m}$.  In particular, every $\mathbb{A}^{1}$-curve carried by $U$ must be contracted by $f$, showing that $U$ is not $\mathbb{A}^{1}$-connected.

(2) $\Rightarrow$ (3):  Suppose that the rays not contained in $D$ are contained in a proper linear subspace $L \subset N$, or in other words, every cone in $\Sigma_{D}$ is contained in $L$.  Then projection perpendicular to $L$ defines a dominant morphism $U \to \mathbb{G}_{m}$.  In particular, $\mathbf k[U]^{\times}$ has a non-trivial function.

(3) $\Rightarrow$ (1): Choose a spanning subset $\{ v_{1},\ldots,v_{n}\}$ of the rays of $\Sigma$ corresponding to divisors not contained in $D$. Let $\Gamma$ be the cone that they span in $N$. The corresponding rational map of fans $\Gamma \to \Sigma_D$ induces a rational toric morphism $\phi: \mathbb{A}^{n} \dashrightarrow U$.  Let $V \subset \mathbb A^n$ be the domain of this rational map. Then it contains the torus and the codimension $1$-torus orbits so the complement has codimension $\geq 2$. A general line on $\mathbb A^n$ avoids this codimension $2$ loci and it goes through two general points of $V$.
This implies that $V$ is $\mathbb A^1$-connected.
Thus $U$ is also $\mathbb{A}^{1}$-connected.

(2) $\Leftrightarrow$ (4): A non-trivial relation $\sum a_{i}D_{i} \sim 0$ yields an equality $\sum a_{i}D_{i} = \mathrm{div}(f)$ and thus an $f \in \mathbf{k}[U]^{\times} \backslash \mathbf{k}^{\times}$, and conversely.

We next turn to separable $\mathbb{A}^{1}$-connectedness:

(a) $\Rightarrow$ (c): If $(X_{\Sigma},D)$ is separably $\mathbb{A}^{1}$-connected, there is a map $f: \mathbb{P}^{1} \to X_{\Sigma}$ through a general point of $X_{\Sigma}$ such that $f^{*}T_{X_{\Sigma}}$ is ample.  This implies that the log tangent bundle admits no trivial quotients.

(b) $\Leftrightarrow$ (c): We have an exact sequence
\begin{equation*}
0 \to \Omega_{X_{\Sigma}} \to M \otimes \mathcal{O}_{X_{\Sigma}} \xrightarrow{\psi} \oplus_{\Delta_{i} \not \subset D} \mathcal{O}_{\Delta_{i}} \to 0
\end{equation*}
where the $i$th factor of $\psi$ is the composition of $\langle -,v_{i} \rangle: M \otimes \mathcal{O}_{X_{\Sigma}} \to \mathbb{Z} \otimes \mathcal{O}_{X_{\Sigma}}$ 
and the quotient map $\mathcal{O}_{X_{\Sigma}} \to \mathcal{O}_{\Delta_{i}}$.
Note that the vectors $\{ v_{i} \}_{\Delta_{i} \not \subset D}$ fail to span $N \otimes \mathbf{k}$ if and only if there is a map $\mathcal{O}_{X_{\Sigma}} \to M \otimes \mathcal{O}_{X_{\Sigma}}$ whose composition with $\psi$ is zero.   
This latter condition is equivalent to saying that there is a non-zero section of $\Omega_{X_{\Sigma}}$, or equivalently, a non-zero map $T_{X_{\Sigma}} \to \mathcal{O}_{X_{\Sigma}}$.

(b) $\Rightarrow$ (a): Just as in the proof of (3) $\Rightarrow$ (1), we obtain a rational map $\phi: \mathbb{A}^{n} \dashrightarrow U$.  We claim that this map is separable; this is clear in characteristic $0$.  When the characteristic $p$ is positive, since we are assuming the rays span $N \otimes \mathbf{k}$ they also span $N \otimes \mathbb{F}_{p}$.  This means that the sublattice $N' \subset N$ spanned by these rays does not have index divisible by $p$, showing that $\phi$ is separable.  Since $\mathbb{A}^{n}$ is separably $\mathbb{A}^{1}$-connected and $\phi$ is well-defined away from a codimension $\geq 2$ subset, we conclude that $U$ is also separably $\mathbb{A}^{1}$-connected.
\end{proof}

When $(X_{\Sigma},D)$ is a toric pair that is not $\mathbb{A}^{1}$-connected, the argument of Proposition \ref{prop:equivconditionstoric} shows that there is a rational toric map $\pi: X_{\Sigma} \dashrightarrow Y_{\Sigma'}$ to a log Calabi-Yau pair that contracts every $\mathbb{A}^{1}$-curve.  In this case, we should not expect the growth rate for integral points to have the same form as for an $\mathbb{A}^{1}$-connected variety (see \cite{Wilsch}).  As in \cite{Santens}, it makes sense to first formulate Manin's conjecture for $\mathbb{A}^{1}$-connected toric pairs before trying to extend to the log Calabi-Yau setting.

\subsection{Arbitrary pairs}
We next discuss possible generalizations of Proposition \ref{prop:equivconditionstoric} to arbitrary SNC pairs $(X,\Delta)$.  The implication (1) $\Rightarrow$ (2) of Proposition \ref{prop:equivconditionstoric} holds for arbitrary pairs using exactly the same argument.

\begin{lemm}
\label{lemm:A1connectedimpliesnoinvertible}
Let $(X,\Delta)$ be an SNC pair with reduced boundary.  If $(X,\Delta)$ is geometrically $\mathbb{A}^{1}$-connected, then $\mathbf k[U]^{\times} = \mathbf k^{\times}$ where $U = X\setminus \mathrm{Supp}(\Delta)$.
\end{lemm}

\begin{coro}
Let $(X,\Delta)$ be a geometrically $\mathbb{A}^{1}$-connected SNC pair with reduced boundary.  
Then the irreducible components $\{ \Delta_{i} \}$ of $\Delta$ are linearly independent in $\Pic(X)_{\mathbb{Q}}$.
\end{coro}

\begin{proof}
We may assume that our ground field is algebraically closed.
Suppose there were a linear relation between the $\Delta_{i}$ in $\Pic(X)_{\mathbb{Q}}$.  After multiplying by a sufficiently divisible positive integer, we obtain a linear relation between Cartier divisors supported on the $\Delta_{i}$'s.  This divisor must have the form $\mathrm{div}(f)$ where $f \in \mathbf k[U]^{\times}$, contradicting $\mathbb{A}^{1}$-connectedness.
\end{proof}

However, the converse of Lemma \ref{lemm:A1connectedimpliesnoinvertible} fails for non-toric pairs.  This is demonstrated by the following example.

\begin{exam} \label{exam:notktimes}
We work in characteristic $0$. Let $\Delta \subset \mathbb{P}_{w,x,y,z}^{3}$ be defined by a general homogeneous equation $f_3(x, y, z) = 0$ of degree $3$.  Thus $\Delta$ is the cone over a smooth plane elliptic curve $C \subset \mathbb{P}^{2}$.  It is clear that $(\mathbb{P}^{3},\Delta)$ is an lc Fano pair.

We first claim that $(\mathbb{P}^{3},\Delta)$ is not $\mathbb{A}^{1}$-connected.  Indeed, since we have a morphism $U \to \mathbb{P}^{2} \backslash C$, it suffices to show that the latter open set is not $\mathbb{A}^{1}$-connected.  This is a consequence of the fact that the log tangent bundle for $(\mathbb{P}^{2},C)$ has degree $0$; thus for every $\mathbb{A}^{1}$-curve on $\mathbb{P}^{2} \backslash C$ the log normal sheaf is not ample, so it cannot be very free.  By \cite[Proposition 4.4]{CLT25}, $\mathbb{P}^{2} \backslash C$ is not $\mathbb A^1$-connected. 

We next claim that $\mathbf k[U]^{\times} = \mathbf{k}^{\times}$.  Suppose that $\mathbf k[U]$ has an invertible function $f/g$ where $f,g$ have degree $\geq 1$, then $g$ would be a power of $f_3$.  Since $\mathbf k[w,x,y,z]$ is a UFD and the numerator $f$ cannot vanish on $U$, we see that $f$ also is a power of $f_3$ so that the function is constant.
\end{exam}

\begin{exam}
\label{rema:counterexample}
Let $X \to \mathbb{P}^{3}$ be the blow-up of the cone point in Example \ref{exam:notktimes}.  Let $\widetilde{\Delta}$ denote the strict transform of $\Delta$ and let $E$ denote the exceptional divisor.  The pair $(X,\widetilde{\Delta} + E)$ is a weak Fano SNC pair with $\mathbf k[U]^{\times} \cong \mathbf k^{\times}$ which is not $\mathbb{A}^{1}$-connected. (\cite[Section 5]{Gongyo12} gives many similar examples of weak Fano lc pairs with pathological behavior.)

However, we do not know of an example of a log Fano SNC pair $(X,\Delta)$ which is not $\mathbb{A}^{1}$-connected but still satisfies the condition $\mathbf k[U]^{\times} = \mathbf k^{\times}$.
\end{exam}

Note that in Example \ref{exam:notktimes} the base $(\mathbb{P}^{2},C)$ is log Calabi-Yau.  Just as in the toric setting, it seems difficult to formulate Manin's conjecture for integral points when every $\mathbb{A}^{1}$-curve is contracted by a fibration.  For this reason, we expect that $\mathbb{A}^{1}$-connectedness should be the right perspective for Manin's conjecture.

A version of the implication (a) $\Rightarrow$ (c) of Proposition \ref{prop:equivconditionstoric} also holds for arbitrary pairs $(X,\Delta)$.

\begin{lemm}
Let $(X,\Delta)$ be an SNC pair with reduced boundary.  If $(X,\Delta)$ is geometrically separably $\mathbb{A}^{1}$-connected, then the only morphism $T_{X} \to \mathcal{O}_{X}$ is the zero morphism. 
\end{lemm}

\begin{proof}
    Suppose that there were a non-zero map $T_{X} \to \mathcal{O}_{X}$.  Restricting to a general very free $\mathbb{A}^{1}$-curve, we see that the restricted log tangent bundle would fail to be ample, a contradiction.
\end{proof}

\cite[Remark 1.10]{JLR26} predicts that a converse implication should also hold.  Suppose that $(X,\Delta)$ is an SNC pair such that $-(K_{\underline{X}} + \Delta)$ is nef.  Suppose that for every quotient $T_{X} \to \mathcal{Q}$ of the log tangent bundle the first Chern class $c_{1}(\mathcal{Q})$ is pseudo-effective and non-zero.  Then $(X,\Delta)$ should be geometrically separably $\mathbb{A}^{1}$-connected.

\section{The structure of the space of rational curves on a smooth projective toric variety}
\label{sec:birationalgeometry}

Let $\mathbf k$ be a field of arbitrary characteristic. Let $\underline{X}$ be a smooth projective toric variety defined over $\mathbf k$ of dimension $n$ with the open orbit $\underline{T}$ such that the full toric boundary divisor $\Delta = \sum_i \Delta_i$ is an SNC divisor. 
We also assume that $\underline{X}$ is split. Let $\alpha$ be a nef class of $1$-cycles on $\underline{X}$. Then we consider the following scheme:
\[
\underline{M}_{\alpha}^\circ = \{ f : \underline{\mathbb P}^1 \to \underline{X} \in \underline{\mathrm{Mor}}(\underline{\mathbb P}^1, \underline{X}, \alpha) \, | \, f(\underline{\mathbb P}^1)\cap \underline{T} \neq \emptyset \}.
\]
This space has been studied by Bourqui:
\begin{theo}[{\cite[Theorem 1.10]{Bou16}}]
    Assume that $\alpha \in \mathrm{Nef}_1(\underline{X})^\circ$. 
    Then $\underline{M}_{\alpha}^\circ$ is geometrically irreducible and has the expected dimension: $$\dim(\underline{M}^{\circ}_{\alpha}) = -K_{\underline{X}}. \alpha + n.$$
\end{theo}

In this paper, we extend the above theorem to arbitrary nef classes, moreover, we analyze the birational geometry of $\underline{M}_{\alpha}^\circ$.  
We will appeal to the description of the moduli space of curves via sections of generators of the Cox ring, see e.g.~\cite{Guest95} or \cite{Cox95}.  
First we prove the following lemma:

\begin{lemm}
\label{lemm:smoothness}
    Let $\alpha \in \mathrm{Nef}_1(\underline{X})$ be any nef class. Then $\underline{M}_{\alpha}^\circ$ is smooth and all irreducible components have the expected dimension.
\end{lemm}

\begin{proof}
    After taking a base change to an algebraic closure, we may assume that $\mathbf k$ is algebraically closed. We claim that for any $$[f : \underline{\mathbb P}^1 \to \underline{X}] \in \underline{M}_{\alpha}^\circ(\mathbf k),$$ we have
    \[
    \rH^1(\underline{\mathbb P}^1, f^*T_{\underline{X}}) = 0.
    \]
    Our assertion follows from this claim.

    Let $X$ be the smooth log scheme associated to $(\underline{X}, \Delta)$. Then we have the exact sequence:
    \[
    0 \to T_X \to T_{\underline{X}} \to \oplus_i \mathcal O_{\Delta_i}(\Delta_i) \to 0,
    \]
    where $T_X$ is the log tangent bundle of the smooth log scheme $X$.
    Since $T_X$ is a locally free sheaf, the pullback gives us an exact sequence
    \[
    0 \to f^*T_X \to f^*T_{\underline{X}} \to \mathcal Q \to 0,
    \]
    where $\mathcal Q$ is a torsion sheaf. Since we have
    \[
    T_X \cong \oplus_{k = 1}^{n} \mathcal O_{X},
    \]
    our claim follows.
\end{proof}

Let $r_i = \Delta_i.\alpha$ and $\bold r = (r_i)$.  (We remind the reader that $\bold r$ must satisfy the balancing condition $\sum r_{i} v_{i} = 0$.)
We consider the following morphism:
\[
\Phi_\alpha : \underline{M}_{\alpha}^\circ \to \underline{H}_{\bold r}:= \prod_{i \in \Sigma^{(1)}} \underline{\mathrm{Hilb}}^{[r_i]}(\underline{\mathbb P}^1), [f : \underline{\mathbb P}^1 \to \underline{X}] \mapsto (f^*\Delta_i).
\]
Regarding this fibration, we first prove the following lemma:

\begin{lemm}
\label{lemm:smoothnessoffibration}
    The image 
    \[
    \Phi_\alpha(\underline{M}_{\alpha}^\circ) \subset \underline{H}_{\bold r},
    \]
    is a Zariski open subset, and the morphism
    \[
    \Phi_\alpha : \underline{M}_{\alpha}^\circ \to \Phi_\alpha(\underline{M}_{\alpha}^\circ),
    \]
    is a smooth morphism whose fibers are $n$-dimensional.
\end{lemm}

\begin{proof} 
We first claim that every geometric fiber of $\Phi_\alpha$ is smooth of dimension $n$. Indeed, a fiber of $\Phi_\alpha$ parametrizes non-degenerate genus $0$ log maps $f : C \to X$ where $C$ has a fixed log structure. (Here $X$ comes with the log structure induced by the full toric boundary.)
Since $f^*T_X \cong \oplus_{k = 1}^n \mathcal O_{\underline{\mathbb{P}^{1}}}$, the obstruction space $\rH^1(\underline{\mathbb P}^1, f^*T_X)$ for log deformations of $f$ vanishes. Hence our claim follows.

Lemma~\ref{lemm:smoothness} shows that $\underline{M}_{\alpha}^{\circ}$ is equidimensional and the previous paragraph shows that the relative dimension of $\Phi_\alpha$ is the same as the difference in dimensions of $\underline{M}_{\alpha}^\circ$ and $\underline{H}_{\bold r}$.  By applying \cite[Lemma 2.1]{DLTT25} we obtain smoothness of $\Phi_{\alpha}$.
\end{proof}

Moreover $\Phi_\alpha$ is isotrivial:

\begin{lemm}
\label{lemm:isotriviality}
    Let $\mathbf k \subset \mathbf k'$ be any extension and fix a $\mathbf k'$-point $x$ in $\Phi_\alpha(\underline{M}_{\alpha}^\circ)$.
    The fiber of $\Phi_\alpha$ over $x$ is isomorphic to $\underline{T}_{\mathbf k'}$.
\end{lemm}

\begin{proof}
    Without loss of generality, we may assume that $\mathbf k = \mathbf k'$. The point $x$ on $\underline{H}_{\bold r}$ is defined by the choice of a section $\mathsf s_i \in H^0(\underline{\mathbb P}^1, \mathcal O(r_i))$ for every $i \in \Sigma^{(1)}$ up to the natural $\underline{\mathbb G}_m^{\Sigma^{(1)}}$-action. However, the corresponding rational curve $f : \underline{\mathbb P}^1 \to \underline{X}$ is determined by $x$ up to the action of the N\'eron-Severi torus $\underline{T}_{\mathrm{NS}}$. 
    Hence the fiber of $\Phi_\alpha$ at $x$ is isomorphic to
    \[
    \underline{\mathbb G}_m^{\Sigma^{(1)}}/\underline{T}_{\mathrm{NS}} \cong \underline{T}.
    \]
    Thus our assertion follows.
\end{proof}

As a corollary, we obtain
\begin{coro}
    Let $\alpha \in \mathrm{Nef}_1(\underline{X})$ be any nef class. 
    Then $\underline{M}_{\alpha}^\circ$ is geometrically irreducible and rational over $\mathbf k$.
\end{coro}
\begin{proof}
    This follows from Lemmas~\ref{lemm:smoothnessoffibration} and \ref{lemm:isotriviality}.
\end{proof}

Finally we will characterize the image
\[
\underline{U}_{\bold r} :=\Phi_\alpha(\underline{M}_{\alpha}^\circ) \subset \underline{H}_{\bold r}.
\]
To this end, we use the Cox ring of $\underline{X}$ and the representation of $\underline{X}$ as the quotient of the universal torsor. Since we assume that $\underline{X}$ is split, $\underline{X}$ can be determined from a fan $\Sigma \subset N$. Let $\Sigma^{(1)}$ be the set of $1$-dimensional rays which corresponds to the set of boundary divisors. 
Then the Cox ring of $\underline{X}$ is the polynomial ring
\[
R = \mathbf k[x_i \, |\, i \in \Sigma^{(1)}].
\]
For each maximal cone $\sigma \in \Sigma$, let
\[
x_{\hat{\sigma}} = \prod_{i \not \in \sigma} x_i,
\]
and define the irrelevant ideal
\[
I(\Sigma) = \langle x_{\hat{\sigma}}\, |\, \sigma \in \Sigma_{\max}\rangle \subset R.
\]
We have a natural morphism
\[
\pi_X : \underline{\mathbb A}^{\Sigma^{(1)}} \setminus \underline{Z(I(\Sigma))} \to \underline{X},
\]
which is a $\underline{T}_{\mathrm{NS}}$-torsor.

Let $w = (w_i)\in \underline{H}_{\bold r}(\overline{\mathbf k})$, and $\mathsf s_i \in H^0(\underline{\mathbb P}^1_{\overline{\mathbf k}}, \mathcal O(r_i))$ be a section corresponding to $w_i$.
Then we have the following lemma:
\begin{lemm}
\label{lemm:Ucharacterization}
    A geometric point $w = (w_i)\in \underline{H}_{\bold r}(\overline{\mathbf k})$ is contained in $\underline{U}_{\bold r}$ if and only if for any geometric point $p \in \underline{\mathbb P^1}$, we have
    \[
    (\mathsf s_i(p)) \not\in \underline{Z(I(\Sigma))}.
    \]
    In other words, a geometric point $w = (w_i)\in \underline{H}_{\bold r}(\overline{\mathbf k})$ is contained in $\underline{U}_{\bold r}$ if and only if for any subset $I \subset \Sigma^{(1)}$, we have
    \begin{equation} \label{eq:supportcondition}
    \cap_{i \in I} \mathrm{Supp}(w_i) \neq \emptyset \implies \cap_{i \in I} \Delta_i \neq \emptyset.
    \end{equation}
\end{lemm}

This lemma allows us to extend the definition of $U_{\bold r} \subset H_{\bold r}$ to arbitrary tuples $\bold r$ (which need not satisfy the balancing condition):

\begin{defi}
    For any tuple $\bold r = (r_{i})_{i \in \Sigma^{(1)}}$ of non-negative integers we define $\underline{U}_{\bold r} \subset \underline{H}_{\bold r}$ to be the open subset whose geometric points satisfy condition \eqref{eq:supportcondition}.
\end{defi}

\section{Counting Campana rational curves on toric varieties}
\label{sec:Campana}
In this section, we prove Manin's conjecture for Campana rational curves on split toric varieties. Over number fields, such a statement has been obtained by Alec Shute and Sam Streeter in \cite{SS24}. The motivic version of this statement is obtained in \cite{Faisant25campana}.
The proof here is inspired by the homological sieve method developed in \cite[Section 8]{DLTT25}.

\subsection{Campana curves and counting}
\label{subsec:Campanacurves}

Let $\mathbf k = \mathbb F_q$ be a finite field and $F = \mathbb F_q(t)$. Let $\underline{X}$ be a smooth projective toric variety defined over $\mathbf k$ of dimension $n$ with the open orbit $\underline{T}$ such that the boundary divisor $\Delta = \sum_i \Delta_i$ is an SNC divisor. We also assume that $\underline{X}$ is split and thus defined by a fan $\Sigma$.
For each $i \in \Sigma^{(1)}$ we choose a positive integer $m_i \in \mathbb Z_{>0}$
and let $(X, \Delta_{\bold m})$ be a Campana orbifold with
\[
\Delta_{\bold m} = \sum_i \left( 1- \frac{1}{m_i}\right)\Delta_i.
\]

Let $\alpha$ be a nef class of $1$-cycles on $\underline{X}$.
As in the previous section, we consider the space of rational curves
\[
\underline{M}_\alpha^\circ,
\]
and its $T$-torsor
\[
\Phi_\alpha : \underline{M}_\alpha^\circ \to \underline{U}_{\bold r}.
\]
\begin{defi}
    A rational curve $[f : \underline{\mathbb P}^1 \to \underline{X}] \in \underline{M}_\alpha^\circ$ is a Campana rational curve if the corresponding divisor $w = (w_i) = \Phi_\alpha([f])$ satisfies
    \[
    w_i \geq m_i w_{i, \mathrm{red}},
    \]
    for every $i$, where $w_{i, \mathrm{red}}$ is the reduced divisor underlying $w_i$. We call these inequalities the Campana conditions.
\end{defi}

Let $\underline{U}_{\bold r, \bold m} \subset \underline{U}_{\bold r}$ be the reduced closed subscheme parametrizing divisors satisfying the Campana conditions.
We define
\[
\underline{M}_{\alpha, \bold m}^\circ = \underline{M}_{\alpha}^\circ \times_{\underline{U}_{\bold r}} \underline{U}_{\bold r, \bold m}.
\]
This can be considered as the space of Campana rational curves with class $\alpha$.  The following lemma reduces point counting questions to the study of Hilbert schemes of points on $\underline{\mathbb{P}}^{1}$:

\begin{lemm}
\label{lemm:countingproduct}
    We have
    \[
    \#\underline{M}_{\alpha, \bold m}^\circ(\mathbf k) = (q-1)^n \#\underline{U}_{\bold r, \bold m}(\mathbf k).
    \]
\end{lemm}
\begin{proof}
    This follows from Lemma~\ref{lemm:isotriviality}.
\end{proof}
We consider the following indicator function $\delta_{\bold m} : \sqcup_{\bold r }H_{\bold r}(\mathbf k) \to \{0, 1\}$:
\[
\delta_{\bold m}(w):= 
\begin{cases}
    1 & \text{ if } w \in \sqcup_{\bold r}U_{\bold r, \bold m} \\
    0 & \text{ otherwise.}
\end{cases}
\]
Note that for $c \in |\mathbb P^1|$, the quantity $\delta_{\bold m}(\bold r c)$ associated to the point $(r_{i}c) \in H_{\bold r |c|}$ does not depend on $c$. Henceforth we simply denote this quantity by $\delta_{\bold m}(\bold r)$.
We have
\[
\#\underline{U}_{\bold r, \bold m}(\bold k) = \sum_{w \in \underline{H}_{\bold r}(\mathbf k)} \delta_{\bold m}(w).
\]
We would like to understand the asymptotic behavior of this counting function. To this end, we consider the following virtual height zeta function: 
\[
\mathsf Z_{\bold m}(\bold t) = \sum_{\bold r} \left(\prod_{i \in \Sigma^{(1)}}q^{-t_ir_i}\right)\#\underline{U}_{\bold r, \bold m}(\bold k) = \sum_{\bold r}\sum_{w \in \underline{H}_{\bold r}(\mathbf k)} \left(\prod_{i\in \Sigma^{(1)}}q^{-t_ir_i}\right)  \delta_{\bold m}(w),
\]
where $\bold r$ runs over all nonnegative tuples.
Since $\delta_{\bold m}$ is multiplicative, the above zeta function is an Euler product: 
\[
\mathsf Z_{\bold m}(\bold t) = \prod_{c \in |\underline{\mathbb P}^1|}\left(\sum_{\bold r}\prod_{i\in \Sigma^{(1)}}q_c^{-t_ir_i}\delta_{\bold m}(\bold r)\right).
\]
We need to understand the pole and its order for this zeta function. To this end, we compare the above virtual height zeta function to the height integral studied in \cite{CLT10} and \cite{PSTVA}.

\subsection{The virtual height zeta functions and the height integrals}
\label{subsec:heights}

Let $\underline{T}(\mathbb A_{F})_{\bold m}$ be the Campana adelic space as in \cite[Section 3.3]{PSTVA}
and denote its indicator function by $\delta_{\bold m} : \underline{T}(\mathbb A_F) \to \{0, 1\}$.
We consider the height integral 
\[
\mathcal I(\delta_{\bold m}; (t_i)) = \int_{\underline{T}(\mathbb A_F)} \mathsf H(\bold t - \bold 1, (g_c))^{-1}\delta_{\bold m}((g_c)) \mathrm d\tau_{\underline{T}}
\]
defined in \cite[Section 4.3.2]{CLT10}. Here $\bold 1 = (1)_{i}$. This becomes the Euler product
\[
\mathcal I(\delta_{\bold m}; (t_i)) = L_*(1, \mathrm{EP}(\underline{T}))^{-1}\prod_{c \in |\underline{\mathbb P}^1|} \mathcal I_c(\delta_{\bold m, c}; (t_i)),
\]
where $\mathcal I_c(\delta_{\bold m, c}; (t_i))$ is defined as
\[
\mathcal I_c(\delta_{\bold m, c}; (t_i)) = \int_{\underline{T}(F_c)} \mathsf H_c(\bold t - \bold 1, g_c)^{-1} \delta_{\bold m, c}(g_c)L_c(1, \mathrm{EP}(\underline{T})) \mathrm d\tau_c.
\]
By Denef's formula (\cite[Proposition 4.5]{CLT10}, see also \cite[Theorem 7.1]{PSTVA}), we have 
\[
\mathcal I_c(\delta_{\bold m, c}; (t_i)) = \left(\frac{q_c}{q_c-1}\right)^n\sum_{A \subset \Sigma^{(1)}} q_c^{-n} \#\underline{\Delta}_A^\circ(\mathbf k_c) \prod_{i \in A}(q_c-1)\frac{q_c^{-m_it_i}}{1-q_c^{-t_i}},
\]
where $n = \dim \, \underline{X}$, $\underline{\Delta}_A^\circ$ is the open stratum of $\underline{\Delta}_A = \cap_{i \in A} \underline{\Delta}_i$.
Since $\underline{\Delta}_A^\circ$ is isomorphic to a split torus, the above local integral becomes
\[
\mathcal I_c(\delta_{\bold m, c}; (t_i)) = \sum_{A \subset \Sigma^{(1)}}  \delta'(A) \prod_{i \in A}\left(\sum_{r_i = m_i}^{\infty}q_c^{-r_it_i}\right),
\]
where $\delta'(A)$ is defined by
\[
\delta'(A) = 
\begin{cases}
1 & \text{ if } \underline{\Delta}_A\neq \emptyset\\
0 & \text{ otherwise}.
\end{cases}
\]
Then $\delta'(A) = 1$ holds if and only if there exists a cone $\sigma \in \Sigma$ such that $A = \sigma^{(1)}$. This observation enables us to show 
\[
\mathcal I_c(\delta_{\bold m, c}; (t_i)) = \sum_{\bold r}\prod_{i \in \Sigma^{(1)}}q_c^{-t_ir_i}\delta_{\bold m}(\bold r).
\]
Using this we have
\begin{prop}
\label{prop:trivialcharacter}
    There exists a positive constant $\epsilon > 0$ such that the function 
    \[
    \left(\prod_{i\in \Sigma^{(1)}}(1-q^{-(m_it_i -1)})\right)\mathsf Z_{\bold m}(\bold t)
    \]
    admits a holomorphic continuation to the domain $\mathsf T_{> -\epsilon}$ defined by $$\Re(t_i)\geq\frac{1}{m_i}- \epsilon,$$ for every $i$, and moreover it satisfies
    \[
    \lim_{t_i \to \frac{1}{m_i}}\left(\prod_{i\in \Sigma^{(1)}}(1-q^{-(m_it_i-1)
})\right)\mathsf Z_{\bold m}(\bold t) = \left(\frac{q}{q-1}\right)^{n} \int_{\underline{X}(\mathbb A_{F})_{\bold m}}\mathsf H(\Delta_m, x) \mathrm d\tau_{\underline{X}},
    \]
    where $\underline{X}(\mathbb A_{F})_{\bold m}$ is the Campana adelic space as in \cite[Section 3.3]{PSTVA} and $\mathsf H(\Delta_m, x)$ is the height function associated to $\Delta_m$. 
    The integral $\int_{\underline{X}(\mathbb A_{F})_{\bold m}}\mathsf H(\Delta_m, x) \mathrm d\tau_{\underline{X}}$ is interpreted as the Tamagawa constant defined in \cite[Section 3.3]{PSTVA}.
\end{prop}

\begin{proof}
   
    For the first statement, it suffices to show that the function
    \[
    \left(\prod_{i \in \Sigma^{(1)}} \zeta_{\underline{\mathbb P}^1}(m_it_i)^{-1}\right) \mathsf Z_{\bold m}(\bold t),
    \]
    is holomorphic in the domain $\mathsf T_{> -\epsilon}$.
     This claim follows from the fact that for a sufficiently small $\epsilon > 0$, there exists $\epsilon' > 0$ such that when $\bold t \in \mathsf T_{> -\epsilon}$, for every $c \in |\underline{\mathbb P}^1|$ we have
    \[
    \left(\prod_{i\in \Sigma^{(1)}} (1-q_c^{-m_it_i})\right)\left(\sum_{\bold r}\prod_{i \in \Sigma^{(1)}}q_c^{-t_ir_i}\delta_{\bold m}(\bold r)\right) = 1 + O(q_c^{-(1 + \epsilon')}).
    \]
    Thus the Euler product for $\left(\prod_{i\in \Sigma^{(1)}} \zeta_{\underline{\mathbb P}^1}(m_it_i)^{-1}\right) \mathsf Z_{\bold m}(\bold t)$ converges for $t_{i}$ in this range.
      This is easy to see from the expression.

 For the second statement, one can compute as
 \begin{align*}
& \lim_{t_i \to \frac{1}{m_i}}\left(\prod_{i\in \Sigma^{(1)}}(1-q^{-(m_it_i-1)
})\mathsf Z_{\bold m}(\bold t)\right) \\  & = \lim_{t_i \to \frac{1}{m_i}}\left(\prod_{i\in \Sigma^{(1)}}(1-q^{-(m_it_i-1)
})\zeta_{\underline{\mathbb P}^1}(m_it_i)\right) L_*(1, \mathrm{EP}(\underline{T}))\prod_{i\in \Sigma^{(1)}} \zeta_{\underline{\mathbb P}^1}(m_it_i)^{-1} \mathcal I(\delta_{\bold m}; (t_i)) \\
& = \left(\frac{q}{q-1}\right)^{n}\int_{\underline{X}(\mathbb A_{F})_{\bold m}}\mathsf H(\Delta_m, x) \mathrm d\tau_{\underline{X}},
 \end{align*}
 where the last equality follows from the proof of \cite[Lemma 9.3]{PSTVA} with $L = -(K_{\underline{X}} + \Delta_{\bold m})$. 
\end{proof}

\subsection{Twisted virtual height zeta functions}

Proposition~\ref{prop:trivialcharacter} is enough to deduce the usual Manin's conjecture for split smooth projective toric varieties over $\mathbb F_q(t)$ (which corresponds to $m_i = 1$ for every $i$).  However, it is not quite enough to deduce Manin's conjecture when $m_i > 1$. 
Indeed, when $m_i = 1$ for every $i$, one may apply the counting argument of \cite[Proposition 8.8]{DLTT25}. However, this argument does not work when $m_i > 1$ for some $i$ and one needs to consider the virtual height zeta functions twisted by characters as below. 

Set $m = \mathrm{lcm}_i \{m_i\}$.
For each $i$ let $b_i$ be an integer such that $0 \leq b_i < m$ and let $\bold b = (b_i)$. We consider the following character $$\chi_{\bold b}(-, \bold r) : \prod_{i \in \Sigma^{(1)}} \mathbb Z/m \mathbb Z  \to \mathbb C^\times$$ defined by
\[
\chi_{\bold b}(\bold a, \bold r) = \prod_i\exp\left(\frac{2\pi\sqrt{-1}a_i(r_i - b_i)}{m_i}\right),
\]
where $\bold a = (a_i) \in \prod_i \mathbb Z/m \mathbb Z$.  Let us fix $\bold a, \bold b$. We define the twisted virtual height zeta function by
\[
\mathsf Z_{\bold a, \bold b}(\bold t) = \sum_{\bold r} \left(\prod_{i\in \Sigma^{(1)}}q^{-t_ir_i}\right)\chi_{\bold b}(\bold a, \bold r)\#\underline{U}_{\bold r, \bold m}(\bold k).
\]
Using multiplicativity, this is equal to
\[
\chi_{\bold b}(\bold a, \bold 0)\prod_{c \in |\underline{\mathbb P}^1|}\left(\sum_{\bold r}\prod_{i \in \Sigma^{(1)}}q_c^{-t_ir_i} \chi_{\bold 0}(\bold a, \bold r)^{|c|}\delta_{\bold m}(\bold r)\right).
\]
As in Section~\ref{subsec:heights}, we express this function as a certain height integral. To this end, we define global semi-characters on $\underline{T}(\mathbb A_F)$.
Let $c \in |\underline{\mathbb P}^1|$. For $g_c \in \underline{T}(F_c)$, 
we consider $g_c$ as a jet on $\underline{X}$. Then 
we define $\widetilde{\chi}_{\bold 0, c}(\bold a, -) : \underline{T}(F_c) \to \mathbb C^\times$ as
\[
\widetilde{\chi}_{\bold 0, c}(\bold a, g_c) =  \chi_{\bold 0}(\bold a, \bold r)^{|c|},
\]
with $r_i = v_c(g_c^*\Delta_i)$ where $v_c$ is the discrete valuation with respect to $c$.
We also define the global semi-character $\widetilde{\chi}_{\bold b}(\bold a, -) : \underline{T}(\mathbb A_F) \to \mathbb C^\times$ by
\[
\widetilde{\chi}_{\bold b}(\bold a, (g_c)) = \chi_{\bold b}(\bold a, \bold 0)\prod_{c \in |\underline{\mathbb P}^1|} \widetilde{\chi}_{\bold 0, c}(\bold a, g_c).
\]
Now we define the twisted height integral by
\[
\mathcal I_{\bold a, \bold b}(\delta_{\bold m}; \bold t) = \int_{\underline{T}(\mathsf A_F)} \widetilde{\chi}_{\bold b}(\bold a, (g_c))\mathsf H(\bold t - \bold 1, (g_c))^{-1}\delta_{\bold m}((g_c))\mathrm d\tau_{\underline{T}}.
\]
This becomes the following Euler product:
\[
\mathcal I_{\bold a, \bold b}(\delta_{\bold m}; \bold t) = \chi_{\bold b}(\bold a, \bold 0) L_*(1, \mathrm{EP}(\underline{T}))^{-1} \prod_{c \in |\underline{\mathbb P}^1|}\mathcal I_{\bold a, c}(\delta_{\bold m, c}; \bold t), 
\]
where the twisted local height integral $\mathcal I_{\bold a, c}(\delta_{\bold m, c}; \bold t)$ is given by
\[
\mathcal I_{\bold a, c}(\delta_{\bold m, c}; \bold t) = \int_{\underline{T}(F_c)} \widetilde{\chi}_{\bold 0, c}(\bold a, g_c)\mathsf H_c(\bold t - \bold 1, g_c)^{-1}\delta_{\bold m, c}(g_c) L_c(1, \mathrm{EP}(\underline{T})) \mathrm d \tau_c.
\]
The computation of \cite[Theorem 7.1]{PSTVA} shows that we have
\[
\mathcal I_{\bold a, c}(\delta_{\bold m, c}; \bold t) = 
\left(\frac{q_c}{q_c-1}\right)^n\sum_{A \subset \Sigma^{(1)}} q_c^{-n} \#\underline{\Delta}_A^\circ(\mathbf k_c) \prod_{i \in A}(q_c-1)\frac{q_c^{-m_it_i}}{1-\chi_{\bold 0}(\bold a, \bold e^i)^{|c|}q_c^{-t_i}},
\]
where $\bold e^i$ is defined by $e^i_i = 1$ and $e^i_j = 0$ for $j \neq i$.
Then as in Section~\ref{subsec:heights}, this becomes
\begin{align*}
\mathcal I_{\bold a, c}(\delta_{\bold m, c}; \bold t)
&= \sum_{A \subset \Sigma^{(1)}}  \delta'(A) \prod_{i \in A}\left(\sum_{r_i = m_i}^{\infty}\chi_{\bold 0}(\bold a, \bold e^i)^{|c|r_i}q_c^{-r_it_i}\right) \\
& = \sum_{\bold r}\prod_{i \in \Sigma^{(1)}}q_c^{-t_ir_i} \chi_{\bold 0}(\bold a, \bold r)^{|c|}\delta_{\bold m}(\bold r).
\end{align*}
Thus we have
\begin{prop}
\label{prop:nontrivialcharacter}
    There exists a positive constant $\epsilon > 0$ such that the function
    \[
    \left(\prod_{i \in \Sigma^{(1)}}(1-q^{-(m_it_i -1)})\right)\mathsf Z_{\bold a, \bold b}(\bold t)
    \]
    admits a holomorphic continuation to the domain $\mathsf T_{> -\epsilon}$, and moreover it satisfies 
    \begin{align*}
    &\lim_{t_i \to \frac{1}{m_i}}\left(\prod_{i\in \Sigma^{(1)}}(1-q^{-(m_it_i-1)
})\right)\mathsf Z_{\bold a, \bold b}(\bold t) \\
&= \chi_{\bold b}(\bold a, \bold 0)\left(\frac{q}{q-1}\right)^{\# \Sigma^{(1)}} \prod_{c \in |\underline{\mathbb P}^1|}\int_{\underline{T}(F_c)_{\bold m}}\widetilde{\chi}_{\bold 0, c}(\bold a, g_c)\mathsf H_c(\Delta_m, g_c)L_c(1, \mathrm{EP}(\underline{X})) \mathrm d\tau_{c}.
    \end{align*}
 We denote this leading constant as $\left(\frac{q}{q-1}\right)^{n} \chi_{\bold b}(\bold a, \bold 0)\tau_{\bold a}$.   
\end{prop}
\begin{proof}
    The proof is similar to the proof of Proposition~\ref{prop:trivialcharacter}.
\end{proof}

Recall that we have
\[
\mathsf Z_{\bold a, \bold b}(\bold t) = \sum_{\bold r} \left(\prod_{i\in \Sigma^{(1)}}q^{-t_ir_i}\right)\chi_{\bold b}(\bold a, \bold r)\#\underline{U}_{\bold r, \bold m}(\bold k).
\]
We fix $\bold b$ and we write $\bold r \equiv \bold b \mod  \bold m$ if $r_i \equiv b_i \mod m_i$ for every $i$. 
Then the orthogonality relation shows that
\[
 \frac{1}{m^{\#\Sigma^{(1)}}}\sum_{\bold a  \in \prod_i \mathbb Z/m \mathbb Z}\mathsf Z_{\bold a, \bold b}(\bold t)
= \sum_{\bold r \equiv \bold b \mod  \bold m} \left(\prod_{i\in \Sigma^{(1)}}q^{-t_ir_i}\right)\#\underline{U}_{\bold r, \bold m}(\bold k) := \mathsf Z_{\bold m, \bold b}(\bold t).
\]
Using this we have the following proposition:
\begin{prop}
\label{prop:counting}
    There exists $\eta > 0$ such that assuming $\bold r \equiv \bold b \mod \bold m$, we have
    \[
    \frac{\#\underline{U}_{\bold r, \bold m}(\bold k)}{q^{\sum_i \frac{r_i}{m_i}}} = \frac{q^n}{(q-1)^nm^{\#\Sigma^{(1)}}}\sum_{\bold a  \in \prod_i \mathbb Z/m \mathbb Z} \chi_{\bold b}(\bold a, \bold 0)\tau_{\bold a} + O(q^{-\eta\min_i\{r_i\}}).
    \]
\end{prop}

\begin{proof}
    The following counting argument comes from \cite[Proposition 8.8]{DLTT25}. 
    We expand the product: 
    \[
   \left( \prod_{i \in \Sigma^{(1)}} (1-q^{-(m_it_i-1)})\right) \mathsf Z_{\bold m, \bold b}(\bold t) = \sum_{\bold j \equiv \bold b} B_{\bold j}\prod_{i\in \Sigma^{(1)}} q^{-j_it_i}.
    \]
    Then it follows from Proposition~\ref{prop:nontrivialcharacter} that this absolutely converges in the domain $\mathsf T_{>-\epsilon}$. This implies that we have
    \[
    B_{\bold j} = o(q^{\sum_i \frac{j_i}{m_i}-\eta\sum_i j_i}),
    \]
    for any $0 < \eta < \epsilon$.
    Now $ \mathsf Z_{\bold m, \bold b}(\bold t)$ can be obtained from the power series $\sum_{\bold j \equiv  \bold b} B_{\bold j}\prod_{i} q^{-j_it_i}$ by multiplying by $\prod_i (1-q^{-(m_it_i-1)})^{-1}$. By comparing coefficients, we have
    \[
    \frac{\#\underline{U}_{\bold r, \bold m}(\bold k)}{q^{\sum_i \frac{r_i}{m_i}}} = \sum_{j_i \leq r_i, \, \bold j \equiv \bold b} B_{\bold j} \prod_{i \in \Sigma^{(1)}} q^{-\frac{j_i}{m_i}}.
    \]
    Thus it follows from Proposition~\ref{prop:nontrivialcharacter} that
    \[
    \lim_{r_i \to \infty}\frac{\#\underline{U}_{\bold r, \bold m}(\bold k)}{q^{\sum_i \frac{r_i}{m_i}}} = \sum_{\bold j\equiv \bold b}B_{\bold j} \prod_{i\in \Sigma^{(1)}} q^{-\frac{j_i}{m_i}} = \frac{q^n}{(q-1)^nm^{\#\Sigma^{(1)}}}\sum_{\bold a  \in \prod_i \mathbb Z/m \mathbb Z} \chi_{\bold b}(\bold a, \bold 0)\tau_{\bold a}.
    \]
    The error term is given by 
    \[
    \sum_i \sum_{j_i > r_i}|B_{\bold j }|\prod_{i'\in \Sigma^{(1)}} q^{-\frac{j_{i'}}{m_{i'}}} = \sum_i O(q^{-\eta r_i}).
    \]
\end{proof}

As a corollary, we have

\begin{coro}
\label{coro:uniformbound}
    There exists some uniform constant $\mathsf C > 0$ only depending on $\underline{X}$ and $\bold m$ such that we have
    \[
    \#\underline{U}_{\bold r, \bold m}(\bold k) \leq \mathsf C q^{\sum_i \frac{r_i}{m_i}}.
    \]
\end{coro}

\begin{proof}
Note that we have $|\tau_{\bold a}| \leq \tau_0$.
Thus our assertion follows from Proposition~\ref{prop:counting}.
\end{proof}

\subsection{The main result}

Recall that the lattice of numerical $1$-cycles $N_1(\underline{X})_{\mathbb Z}$ sits in the exact sequence:
\[
0 \to N_1(\underline{X})_{\mathbb Z} \to \mathbb Z^{\Sigma^{(1)}} \to \mathfrak X(\underline{T})^\vee \to 0,
\]
where $\mathfrak X(\underline{T})$ is the group of characters of $T$.

Next for our Campana orbifold $(X, \Delta_{\bold m})$, we define the following index:
\begin{defi}
    We define the index $\mathsf r(X, \Delta_{\bold m})$ by
    \[
    \mathsf r(X, \Delta_{\bold m}) = \min\{ -(K_{\underline{X}} + \Delta_{\bold m}). \alpha > 0 \, | \, \alpha \in N_1(\underline{X})_{\mathbb Z} \}.
    \]
\end{defi}

Now we define the counting function we are interested in:
for any positive integer $d$, we define
\[
N((X, \Delta_{\bold m}), \mathsf r(X, \Delta_{\bold m})d)  = \sum_{\alpha \in \mathrm{Nef}_1(\underline{X})_{\mathbb Z}, -(K_{\underline{X}} + \Delta_{\bold m}).\alpha = \mathsf r(X, \Delta_{\bold m})d} \#M_{\alpha, \bold m}^\circ(\bold k).
\]
We are interested in the asymptotic behavior of this counting function as $d$ goes to $\infty$. To this end, we consider the following auxiliary cones: let $\epsilon > 0$ be a small positive rational number. We define the shrunken nef cone by
\[
\mathrm{Nef}_1(\underline{X})_\epsilon = \{ \alpha \in \mathrm{Nef}_1(\underline{X})\, | \, \Delta_i.\alpha \geq -\epsilon(K_{\underline{X}} + \Delta_{\bold m}).\alpha \text{ for any $i$} \},
\]
and we denote the closure of the complement by 
\[
\mathsf C_\epsilon = \overline{\mathrm{Nef}_1(\underline{X}) \setminus \mathrm{Nef}_1(\underline{X})_\epsilon}.
\]
Note that these are rational polyhedral cones.

\begin{defi}
    Let $\mathsf C \subset \mathrm{Nef}_1(\underline{X})$ be a rational polyhedral subcone. We fix the Lebesgue measure on $N_1(\underline{X})$ such that the fundamental domain of the lattice $N_1(\underline{X})_{\mathbb Z}$ has volume $1$.
    We define the alpha constant $\alpha(\mathsf C)$ by
    \[
    \alpha(\mathsf C) = (\dim \, N_1(\underline{X})) \cdot \mathsf r(X, \Delta_{\bold m}) \cdot \mathrm{vol}(\{\alpha \in \mathsf C \, | \, -(K_{\underline{X}} + \Delta_{\bold m}).\alpha \leq 1 \}).
    \]
    Note that we may interpret this as $\mathsf r(X, \Delta_{\bold m})$ times the volume of
    \[
    \mathsf C\cap \{ -(K_{\underline{X}} + \Delta_{\bold m}).\alpha = 1\}.
    \]
    The constant $\alpha(\mathrm{Nef}_1(\underline{X}))$ is denoted by $\alpha(X, \Delta_{\bold m})$
\end{defi}
Finally we define the following counting functions: let $\mathsf C \subset \mathrm{Nef}_1(\underline{X})$ be a rational polyhedral subcone.
Then
\[
N(\mathsf C, \mathsf r(X, \Delta_{\bold m})d) = \sum_{\alpha \in \mathsf C_{\mathbb Z}, -(K_{\underline{X}} + \Delta_{\bold m}).\alpha = \mathsf r(X, \Delta_{\bold m})d} \#M_{\alpha, \bold m}^\circ(\bold k).
\]
Now we state our main theorem: 
\begin{theo}
\label{theo:Campana}
    We have
 \begin{align*}
   &\frac{1}{m}\sum_{\ell = 0}^{m-1} \frac{1}{q^{\mathsf r(X, \Delta_{\bold m})\ell}}N((X, \Delta_{\bold m}), \mathsf r(X, \Delta_{\bold m})(md+\ell)) \\ &\sim \mathsf c(X, \Delta_{\bold m})q^{\mathsf r(X, \Delta_{\bold m})md}(\mathsf r(X, \Delta_{\bold m})md)^{\rho(\underline{X})-1},
    \end{align*}
    as $d \to \infty$ where $\mathsf c(X, \Delta_{\bold m})$ is given by
    \[
     \mathsf c(X, \Delta_{\bold m}) = \frac{\alpha(X, \Delta_{\bold m})q^n}{\prod_{i\in \Sigma^{(1)}} m_i}\sum_{\bold a  \in \mathcal X(\underline{T})_{\bold m}}\tau_{\bold a}.
    \]
    Here $\mathcal X(\underline{T})_{\bold m} \subset \prod_i \mathbb Z/m_i\mathbb Z$ is the orthogonal dual of the image of
    \[
    N_1(\underline{X})_{\mathbb{Z}} \hookrightarrow \mathbb Z^{\Sigma^{(1)}} \to \prod_{i\in \Sigma^{(1)}} \mathbb Z/m_i\mathbb Z.
    \]
    
\end{theo}

\begin{rema}
    In Theorem~\ref{theo:Campana}, we take the average of the counting functions. As pointed out by Ramin Takloo-Bighash, without taking the average, the asymptotic formula behaves periodically and one needs to consider the residue of $d$ modulo $m$. 
    This is already evident in the function field version of Malle's conjecture, see \cite{SantensMalle}.
\end{rema}
\begin{rema}
    \label{rema:characters}
    Since $\mathcal X(T)_{\bold m}$ annihilates curve classes, the pairing
    \[
    \mathbb Z^{\Sigma^{(1)}} \times \mathcal X(T)_{\bold m} \to \mathbb Z/m\mathbb Z
    \]
    descends to
    \[
    X_*(\underline{T}) \times \mathcal X(T)_{\bold m} \to \mathbb Z/m\mathbb Z.
    \]
    In particular, this implies that we have the inclusion
    \[
    \mathcal X(T)_{\bold m} \hookrightarrow \mathrm{Hom}(X_*(\underline{T}), \mathbb Q/\mathbb Z) = X^*(\underline{T})\otimes \mathbb Q/\mathbb Z.
    \]
    Thus when $\bold a  \in \mathcal X(T)_{\bold m}$, $\widetilde{\chi}_{\bold 0, c}(\bold a, g_c)$ is an automorphic character.
\end{rema}

In the next subsection, we will show that $\mathsf c(X, \Delta_{\bold m})$ is a positive real number.
The rest of this section is devoted to the proof of Theorem~\ref{theo:Campana}.
First we prove the following proposition:

\begin{prop}
\label{prop:maincounting}
     We have
    \begin{align*}
   &\frac{1}{m}\sum_{\ell = 0}^{m-1} \frac{1}{q^{\mathsf r(X, \Delta_{\bold m})\ell}}N(\mathrm{Nef}_1(\underline{X})_\epsilon, \mathsf r(X, \Delta_{\bold m})(md + \ell))\\ & \sim \mathsf c(\mathrm{Nef}_1(\underline{X})_\epsilon)q^{\mathsf r(X, \Delta_{\bold m})md}(\mathsf r(X, \Delta_{\bold m})md)^{\rho(\underline{X})-1},
    \end{align*}
    as $d \to \infty$ where $\mathsf c(\mathrm{Nef}_1(\underline{X})_\epsilon)$ is given by
    \[
     \mathsf c(\mathrm{Nef}_1(\underline{X})_\epsilon) = \frac{\alpha(\mathrm{Nef}_1(\underline{X})_\epsilon)q^n}{\prod_{i\in \Sigma^{(1)}} m_i}\sum_{\bold a  \in \mathcal X(\underline{T})_{\bold m}}\tau_{\bold a}.
    \]
\end{prop}

\begin{proof}
    First let us fix some notation: we define the convex set
    \[
    \mathsf R_d = \mathrm{Nef}_1(\underline{X})_\epsilon \cap \{ -(K_{\underline{X}} + \Delta_{\bold m}).\alpha = \mathsf r(X, \Delta_{\bold m})d\},
    \]
    and we denote its set of lattice points by $\mathsf R_{d, \mathbb Z}$.
    We pick one lattice point $\bold r_{d, 0} \in \mathsf R_{d, \mathbb Z}$. (Note that such a lattice point exists assuming $d$ is sufficiently large.)

    With these notations, we have
    \begin{align*}
     &\frac{1}{m} \sum_{\ell = 0}^{m-1} \frac{1}{q^{\mathsf r(X, \Delta_{\bold m})\ell}}N(\mathrm{Nef}_1(\underline{X})_\epsilon, \mathsf r(X, \Delta_{\bold m})(md+ \ell))\\ &= 
       \frac{1}{m} \sum_{\ell = 0}^{m-1} \sum_{\alpha \in \mathsf R_{md + \ell, \mathbb Z}} \frac{1}{q^{\mathsf r(X, \Delta_{\bold m})\ell}}\#M_{\alpha, \bold m}^\circ(\mathbf k).
    \end{align*}
    Then it follows from Lemma~\ref{lemm:countingproduct} that we have 
    \begin{align*}
   & \frac{1}{m} \sum_{\ell = 0}^{m-1} \frac{1}{q^{\mathsf r(X, \Delta_{\bold m})\ell}}
    N(\mathrm{Nef}_1(\underline{X})_\epsilon, \mathsf r(X, \Delta_{\bold m})(md+\ell))\\  &= \frac{1}{m}\sum_{\ell = 0}^{m-1} \sum_{\bold r \in \mathsf R_{md + \ell, \mathbb Z}}\frac{(q-1)^n}{q^{\mathsf r(X, \Delta_{\bold m})\ell}} \#U_{\bold r, \bold m}(\mathbf k).
    \end{align*}
    Let us consider the following lattice:  
    \[
    R = \{ \bold r \in N_1(\underline{X})_{\mathbb Z} \, | \, \sum_i r_i/m_i = 0\}.
    \]
    Then the mapping
    \[
    \mathsf R_{d} \to R_{\mathbb R}, \bold r \mapsto \bold r - \bold r_{d, 0}
    \]
    sends lattice points to lattice points.
    We fix a fundamental set $\Lambda \subset R$ of $R\otimes \mathbb Z/m\mathbb Z$, i.e., if we denote a basis for $R$ by $e_1, \cdots, e_s$, then 
    \[
    \Lambda = \left\{ \sum_{j = 1}^s k_je_j \in R \, \middle| \, 0 \leq k_j < m \right\}
    \]
    Let $n_{d'}$ be the number of translations of $\Lambda$ by $mR$ which are contained in $\mathsf R_{d', \mathbb Z}$ by the above mapping. 
    Then by combining Proposition~\ref{prop:counting} with a lattice counting argument, we obtain:
    \begin{align*}
    & \frac{1}{m} \sum_{\ell = 0}^{m-1} \frac{1}{q^{\mathsf r(X, \Delta_{\bold m})\ell}}
    N(\mathrm{Nef}_1(\underline{X})_\epsilon, \mathsf r(X, \Delta_{\bold m})(md+\ell)))\\ &= \frac{1}{m}\sum_{d' = md}^{m(d+1)-1} n_{d'}\sum_{\bold b \in R_{d'}\otimes \mathbb Z/m\mathbb Z}q^{\mathsf r(X, \Delta_{\bold m})md + n} \frac{1}{m^{\#\Sigma^{(1)}}}\sum_{\bold a  \in \prod_i \mathbb Z/m \mathbb Z} \chi_{\bold b}(\bold a, \bold 0)\tau_{\bold a} \\
    &+ O(q^{(1-\epsilon\eta)\mathsf r(X, \Delta_{\bold m})md}d^{\rho(\underline{X})-1})
    + O(q^{\mathsf r(X, \Delta_{\bold m})md}d^{\rho(\underline{X})-2}).
    \end{align*}
    Here the first error term is coming from Proposition~\ref{prop:counting} and the second error term is coming from the fact that
    \[
    \#\{ \bold r \in \mathsf R_{d', \mathbb Z} \,| \, \bold r \equiv \bold b \} = n_{d'} + O((d')^{\rho(\underline{X})-2}).
    \]
    As the error term is asymptotically negligible, we focus on the main term. 
     Note that we have
    \[
     n_{d'}m^{\dim \, N_1(X)-1} \sim \alpha(\mathrm{Nef}_1(\underline{X})_\epsilon)(\mathsf r(X, \Delta_{\bold m})md)^{\rho(X)-1},
    \]
    The above main term can be transformed to 
    \begin{align*}
    &\alpha(\mathrm{Nef}_1(\underline{X})_\epsilon)(\mathsf r(X, \Delta_{\bold m})d)^{\rho(X)-1}q^{\mathsf r(X, \Delta_{\bold m})md + n} \frac{1}{m^{\#\Sigma^{(1)} + 1}}\\
    &\sum_{\bold a  \in \prod_i \mathbb Z/m \mathbb Z}\sum_{\bold b \in N_1(X)\otimes \mathbb Z/m\mathbb Z} \chi_{\bold b}(\bold a, \bold 0)\tau_{\bold a} 
    \end{align*}
    Then when $\bold a \equiv \bold a' \mod \bold m$, we have 
    \[
    \chi_{\bold b}(\bold a, 0) \tau_{\bold a} = \chi_{\bold b}(\bold a', 0) \tau_{\bold a'}.
    \]
    Thus the above summation becomes
     \begin{align*}
     &\alpha(\mathrm{Nef}_1(\underline{X})_\epsilon)(\mathsf r(X, \Delta_{\bold m})d)^{\rho(X)-1}
    q^{\mathsf r(X, \Delta_{\bold m})md + n} \frac{1}{m\prod_i m_i}\\
    &\sum_{\bold a  \in \prod_i \mathbb Z/m_i \mathbb Z}\sum_{\bold b \in N_1(X)\otimes \mathbb Z/m\mathbb Z} \chi_{\bold b}(\bold a, \bold 0)\tau_{\bold a} 
    \end{align*}
    Then by the orthogonality relation for characters, this becomes
    \begin{align*}
    \alpha(\mathrm{Nef}_1(\underline{X})_\epsilon)(\mathsf r(X, \Delta_{\bold m})md)^{\rho(X)-1}q^{\mathsf r(X, \Delta_{\bold m})md + n} \frac{1}{\prod_i m_i}\sum_{\bold a  \in \mathcal X(\underline{T})_{\bold m}}\tau_{\bold a}.
    \end{align*}
 as $d \to \infty$.
 Rearranging, we get the constant in the theorem statement.
\end{proof}

\begin{proof}[Proof of Theorem~\ref{theo:Campana}]

Let 
\[
\widetilde{N}((X, \Delta_{\bold m}), d) = \frac{1}{m}\sum_{\ell = 0}^{m-1} \frac{1}{q^{\mathsf r(X, \Delta_{\bold m})\ell}}N((X, \Delta_{\bold m}), \mathsf r(X, \Delta_{\bold m})(md+\ell)),
\]
and so on.
    We consider the following inequalities
    \begin{align*}
    &\frac{\widetilde{N}(\mathrm{Nef}_1(\underline{X})_\epsilon, d)}{q^{\mathsf r(X, \Delta_{\bold m})md}(\mathsf r(X, \Delta_{\bold m})md)^{\rho(X)-1}}\leq \frac{\widetilde{N}((X, \Delta_{\bold m}), d)}{q^{\mathsf r(X, \Delta_{\bold m})md}(\mathsf r(X, \Delta_{\bold m})md)^{\rho(X)-1}}\\
    &\leq
    \frac{\widetilde{N}(\mathrm{Nef}_1(\underline{X})_\epsilon, d)}{q^{\mathsf r(X, \Delta_{\bold m})md}(\mathsf r(X, \Delta_{\bold m})md)^{\rho(X)-1}} + \frac{\widetilde{N}(\mathsf C_\epsilon, d)}{q^{\mathsf r(X, \Delta_{\bold m})md}(\mathsf r(X, \Delta_{\bold m})md)^{\rho(X)-1}}.
    \end{align*}
    By Corollary~\ref{coro:uniformbound} 
    \[
    \limsup_{d \to \infty} \frac{\widetilde{N}(\mathsf C_\epsilon, d)}{q^{\mathsf r(X, \Delta_{\bold m})md}(\mathsf r(X, \Delta_{\bold m})md)^{\rho(X)-1}} \leq \frac{(q^m-1)(q-1)^n\mathsf C}{(q-1)}\alpha(\mathsf C_\epsilon).
    \]
    Thus by Proposition~\ref{prop:maincounting}, we have 
     \begin{align*}
    &\mathsf c(\mathrm{Nef}_1(\underline{X})_\epsilon) \leq \liminf_{d \to \infty}\frac{\widetilde{N}((X, \Delta_{\bold m}), d)}{q^{\mathsf r(X, \Delta_{\bold m})md}(\mathsf r(X, \Delta_{\bold m})md)^{\rho(X)-1}}\\
    &\leq \limsup_{d \to \infty}\frac{\widetilde{N}((X, \Delta_{\bold m}), d)}{q^{\mathsf r(X, \Delta_{\bold m})md}(\mathsf r(X, \Delta_{\bold m})md)^{\rho(X)-1}}\leq\mathsf c(\mathrm{Nef}_1(\underline{X})_\epsilon) + \frac{(q^m-1)(q-1)^n\mathsf C}{(q-1)}\alpha(\mathsf C_\epsilon).  
    \end{align*}
    As $\epsilon \to 0$, our assertion follows.
\end{proof}

\subsection{The leading constant}
\label{subsec:leadingconstandforCampana}
In this subsection, we discuss the leading constant $\mathsf c(X, \Delta_{\bold m})$.

\subsubsection*{Positivity}

We show that this number is positive following the ideas of \cite[Theorem 7.1]{CLTBT}:

\begin{prop}
    The leading constant $\mathsf c(X, \Delta_{\bold m})$ is positive.
\end{prop}

\begin{proof}
    It suffices to show that the constant
    \[
    \mathsf c' = \sum_{\bold a \in \mathcal X(\underline{T})_{\bold m}} \tau_{\bold a},
    \]
    is positive. From the definition, we obtain 
    \[
    \mathsf c' = \lim_{t_i \mapsto 1/m_i}\left(\prod_{i \in \Sigma^{(1)}}(1-q^{-(m_it_i-1)}) \right) \sum_{\bold a \in \mathcal X(\underline{T})_{\bold m}} \int_{\underline{T}(\mathbb A_F)} \widetilde{\chi}_{\bold 0}(\bold a, (g_c)) \mathsf H(\bold t - \bold 1, (g_c))^{-1}\delta_{\bold m}((g_c)) \mathrm d\tau_{\underline{T}}.
    \]
    By Remark~\ref{rema:characters} and the orthogonality relation, this is equal to 
    \begin{align*}
        \lim_{t_i \mapsto 1/m_i}\left(\prod_{i \in \Sigma^{(1)}}(1-q^{-(m_it_i-1)})\right)\left( \# \mathcal X(\underline{T})_{\bold m} \right) \int_{\underline{T}(\mathbb A_F)^{\mathcal X(\underline{T})_{\bold m}}}  \mathsf H(\bold t - \bold 1, (g_c))^{-1}\delta_{\bold m}((g_c)) \mathrm d\tau_{\underline{T}},
    \end{align*}
    where $\underline{T}(\mathbb A_F)^{\mathcal X(\underline{T})_{\bold m}}$ is given by
    \[
    \underline{T}(\mathbb A_F)^{\mathcal X(\underline{T})_{\bold m}} = \{ (g_c) \in T(\mathbb A_F) \, | \, \text{ for every $\bold a \in \mathcal X(\underline{T})_{\bold m}$, $\widetilde{\chi}_{\bold 0}(\bold a, (g_c)) = 1$}\}.
    \]
    To this end, we consider the adelic space of Darmon points, i.e., 
    \[
    \underline{T}(\mathbb A_F)_{\bold m}' = \prod_c' \underline{T}(F_c)_{\bold m}',
    \]
    where $\underline{T}(F_c)_{\bold m}'$ is given by
    \[
    \underline{T}(F_c)_{\bold m}' = \left\{g_c \in \underline{T}(F_c)\, \left|\, \text{for every $i$, $m_i$ divides $v_c(g_c^*\Delta_i)$} \right. \right\} .
    \]
    Since we have 
    \[
     \underline{T}(\mathbb A_F)_{\bold m}' \subset \underline{T}(\mathbb A_F)^{\mathcal X(\underline{T})_{\bold m}},
    \]
    it suffices to show that  
    \begin{align*}
        \lim_{t_i \mapsto 1/m_i}\prod_i(1-q^{-(m_it_i-1)}) \int_{\underline{T}(\mathbb A_F)_{\bold m}'}  \mathsf H(\bold t - \bold 1, (g_c))^{-1}\delta_{\bold m}((g_c)) \mathrm d\tau_{\underline{T}},
    \end{align*}
    is positive. It turns out that this is equal to
    \[
    \int_{\underline{X}(\mathbb A_F)_{\bold m}'} \mathsf H(\Delta_{\bold m}, (x_c)) \mathrm d\tau_{\underline{X}},
    \]
    where $\underline{X}(\mathbb A_F)_{\bold m}'$ is given by
    \[
    \underline{X}(\mathbb A_F)_{\bold m}' = \prod_c \overline{\underline{T}(F_c)_{\bold m}'} \subset \prod_c \underline{X}(F_c).
    \]
    One can prove that this is positive using the analogue of Denef's formula for Darmon points. Thus our assertion follows.
\end{proof}

\subsubsection*{The conjecture in \cite{CLTBT}}

Here we briefly explain why our constant $\mathsf c(X, \Delta_{\bold m})$ is compatible with the conjectural description of the leading constant in \cite[Conjecture 8.3]{CLTBT}. 
We freely use the notation in \cite[Section 8]{CLTBT}.
First of all, note that the set
\[
\{ \widetilde{\chi}_{\bold 0}(\bold a, -)\, |\, \bold a \in \mathcal X(\underline{T})_{\bold m}\},
\]
is exactly equal to the set of certain unramified automorphic characters:
\[
\left\{ \chi : \underline{T}(\mathbb A_F)/(\mathsf K \cdot \underline{T}(F)) \to S^1 \, \left|\, \begin{array}{c} \text{$\chi$ is a continuous homomorphism} \\ \text{and for any $i$, $\chi_i^{m_i} = 1$} \end{array} \right. \right\},
\]
where $\mathsf K = \prod_c \underline{T}(\mathfrak o_c)$ and $\chi_i$ is a unramified Hecke character associated to $\chi$ and $i$.
We denote this group by $(\underline{T}(\mathbb A_F)/(\mathsf K \cdot \underline{T}(F)))^\vee_{\bold m}$.
We have the following lemma:
\begin{lemm}[{\cite[Corollary 4.6 and Lemma 4.7]{Loughran}}]
\label{lemm:Loughran}
    We have the following isomorphism:
    \[
    \mathrm{Br}_1(\underline{X}, \Delta_{\bold m})/\mathrm{Br}(F) \cong (\underline{T}(\mathbb A_F)/\underline{T}(F))^\vee_{\bold m},
    \]
    where $\mathrm{Br}_1(\underline{X}, \Delta_{\bold m})$ is the algebraic Campana Brauer group defined in \cite[Definition 8.1]{CLTBT}.
    Moreover this induces a bilinear pairing
    \[
    \mathrm{Br}_1(\underline{X}, \Delta_{\bold m})/\mathrm{Br}(F) \times (\underline{T}(\mathbb A_F)/\underline{T}(F))_{\bold m} \to S^1.
    \]
\end{lemm}

 Since \cite{Loughran} focuses on the case of number fields, we will verify the necessary adjustments for the function field case in Section~\ref{sec:appendix}.
By Lemma \ref{lemm:Loughran} we have
\[
\mathrm{Br}_1(\underline{X}, \Delta_{\bold m})^{\mathsf K}/\mathrm{Br}(F) \cong (\underline{T}(\mathbb A_F)/(\mathsf K \cdot \underline{T}(F)))^\vee_{\bold m},
\]
where $\mathrm{Br}_1(\underline{X}, \Delta_{\bold m})^{\mathsf K}$ is its subgroup consisting of elements which are trivial on $\mathsf K$.
Since the height function $\mathsf H$ and $\delta_{\bold m}$ are $\mathsf K$-invariant, by arguing as in \cite[Theorem 8.6]{CLTBT}, we obtain
\[
\sum_{b \in \mathrm{Br}_1(\underline{X}, \Delta_{\bold m})/\mathrm{Br}(F)} \hat{\tau}_{X, \Delta_{\bold m}}(b) = \sum_{b \in \mathrm{Br}_1(\underline{X}, \Delta_{\bold m})^{\mathsf K}/\mathrm{Br}(F)} \hat{\tau}_{X, \Delta_{\bold m}}(b).
\]
Finally for $\bold a \in \mathcal X(\underline{T})_{\bold m}$ and the corresponding $b \in \mathrm{Br}_1(\underline{X}, \Delta_{\bold m})^{\mathsf K}/\mathrm{Br}(F)$, we have $\tau_{\bold a} = \hat{\tau}_{X, \Delta_{\bold m}}(b)$. This explains the compatibility.

\section{Counting $\mathbb A^1$-curves on toric varieties}
\label{sec:A1}

Our next goal is to describe the counting function for $\mathbb{A}^{1}$-curves on toric varieties.  

\subsection{$\mathbb A^1$-curves on toric varieties}
We recall the set up.
Let $\mathbf k = \mathbb F_q$ be a finite field and $F = \mathbb F_q(t)$. Let $\underline{X}$ be a smooth projective toric variety defined over $\mathbf k$ of dimension $n$ with the open orbit $\underline{T}$ such that the full toric boundary divisor $\Delta = \sum_i \Delta_i$ is an SNC divisor. 
We also assume that $\underline{X}$ is split. Let $$D = \sum_{i \in \mathcal A} \Delta_i \leq \Delta,$$ be a reduced boundary divisor and let $X$ be the log scheme associated to $(\underline{X}, D)$. Set $X^\circ = X \setminus \mathrm{Supp}(D)$ and we assume that $X^\circ$ is geometrically $\mathbb A^1$-connected.

Let $\alpha$ be a nef class of $1$-cycles on $\underline{X}$.
As in Section~\ref{sec:birationalgeometry}, we consider the space of rational curves
\[
\underline{M}_\alpha^\circ,
\]
and its $T$-torsor
\[
\Phi_\alpha : \underline{M}_\alpha^\circ \to \underline{U}_{\bold r}.
\]
It will be convenient to define $\mathbb A^1$-curves on $(X, D)$ in terms of their intersections against the boundary divisors:
\begin{defi}
    A rational curve $[f : \underline{\mathbb P}^1 \to \underline{X}] \in \underline{M}_\alpha^\circ$ is an $\mathbb A^1$-curve in $X^\circ$ if the corresponding divisor $w = (w_i) = \Phi_\alpha([f])$ satisfies that for any $i$ with $\Delta_i \subset D$, $w_i$ is either supported at $\{\infty\}$ or empty. We call these conditions $\mathbb A^1$-conditions. Conversely if we have a morphism $f : \underline{\mathbb A}^1 \to X^{\circ}$ with $f(\underline{\mathbb A}^1) \cap \underline{T} \neq \emptyset$, its closure defines an $\mathbb A^1$-curve for some nef class $\alpha$.
\end{defi}

Let $\underline{U}_{\bold r, D} \subset \underline{U}_{\bold r}$ be the reduced closed subscheme parametrizing divisors satisfying the $\mathbb A^1$-conditions.
We define
\[
\underline{M}_{\alpha, D}^\circ = \underline{M}_{\alpha}^\circ \times_{\underline{U}_{\bold r}} \underline{U}_{\bold r, D},
\]
and consider this scheme as the space of $\mathbb A^1$-curves of the class $\alpha$ on $(X, D)$.

The following definition is a key to the study of log Manin's conjecture for integral points:

\begin{defi}\label{defi:clemens-complex}
Let $(\underline{X},D)$ be a split SNC pair.  Write $\mathcal{A}$ for the finite set indexing the irreducible components of $D$.  The (geometric) Clemens complex of the pair is the poset whose elements have the form $(A,Z)$ where $A \subset \mathcal{A}$ and $Z$ is a non-empty irreducible component of the intersection $\cap_{i \in A} D_{i}$.  
We impose the order $(A_{1},Z_{1}) \leq (A_{2},Z_{2})$ if $A_{1} \subset A_{2}$ and $Z_{1} \supset Z_{2}$.
Then we define the dimension of $(A, Z)$ as the dimension of the poset $[(\emptyset, X), (A, Z)]$.

We include the formal pair $(\emptyset, \underline{X})$ as the unique minimal element of the Clemens complex.
\end{defi}

Note that when $(\underline{X},D)$ is a split toric pair, each non-empty intersection $\cap_{i \in A} D_{i}$ is automatically irreducible.  
Hence $A$ and $Z$ determine each other. 
Recall the fan $\Sigma_X$ from the end of \S \ref{sss:A1-curve},
consisting of cones in $\Sigma$ spanned by rays corresponding to irreducible components of $D$.
Alternatively, $\Sigma_X$ can be obtained by removing cones in $\Sigma$ that contain rays in $\Sigma_D$. 
(For the definition of $\Sigma_D$, see the proof of Proposition~\ref{prop:equivconditionstoric}.) 
We may identify the Clemens complex of $X$ with the set of cones in $\Sigma_X$, and identify its  partial order with the inclusions of cones in  $\Sigma_X$.

\begin{defi}
Let $(X,D)$ be a geometrically $\mathbb A^1$-connected smooth projective split toric pair.  For each element $(A,Z)$ in the Clemens complex, we define the face $\mathcal{F}_{A,Z}$ of $\Nef_{1}(X)$ to be the set of numerical classes which have vanishing intersection against $D_{i}$ for every $i \in \mathcal{A} \setminus A$.
\end{defi}

The following shows that the possible numerical classes of $\mathbb{A}^{1}$-curves for $(X,D)$ are controlled by the Clemens complex.

\begin{prop}
\label{prop:numericalclasswithA1}
Let $(X,D)$ be a geometrically $\mathbb A^1$-connected smooth projective split toric pair and fix $\alpha \in \Nef_{1}(X)_{\mathbb Z}$. Then the moduli space $M_{\alpha, D}^\circ$ parametrizing $\mathbb{A}^{1}$-curves of class $\alpha$  is non-empty if and only if $\alpha$ lies in $\mathcal{F}_{A,Z}$ for some $(A, Z)$ in the Clemens complex.

Furthermore, if $M_{\alpha, D}^\circ$ is non-empty then it is irreducible and has dimension $-K_X.\alpha + n = -(K_{\underline{X}} + D).\alpha + n$.
\end{prop}

\begin{proof}
We may assume that the ground field $\mathbf k$ is an algebraically closed field.
If there is an $\mathbb{A}^{1}$-curve of class $\alpha$, then $f: \underline{\mathbb{P}}^{1} \to \underline{X}$ must take $f(\infty)$ to some strata of the boundary $D$, thus identifying an element $(A,Z)$ in the Clemens complex.  (When $f(\infty)$ is in $\underline{X}^\circ$, $A$ is empty.) Note that the curve $f(\underline{\mathbb{P}}^{1})$ must be disjoint from every irreducible component of $D$ not containing $Z$; thus $\alpha$ lies in $\mathcal{F}_{A,Z}$. 
Conversely, if $\alpha \in \mathcal F_{A, Z}$, then by Lemma~\ref{lemm:Ucharacterization}, $\underline{U}_{\bold r, D}$ is non-empty.

If $M_{\alpha, D}^\circ$ is non-empty, then $\underline{U}_{\bold r, D} $ is irreducible and has dimension $-K_X.\alpha$. Thus the last claim follows from Lemma~\ref{lemm:isotriviality}.
\end{proof}

Note that for a geometrically $\mathbb{A}^{1}$-connected smooth projective split toric pair $(X,D)$, the divisor $-K_{X} = -(K_{\underline{X}}+D)$ need not be big.  Nevertheless, one should expect a Northcott property for $\mathbb{A}^{1}$-curves with respect to this polarization.  The required positivity is provided by the following result.

\begin{prop}
\label{prop:northcott}
    Let $(X,D)$ be a geometrically $\mathbb A^1$-connected smooth projective split toric pair.  Let $\mathcal{F}$ be the face of the nef cone of curves perpendicular to $-K_{X}$.  Then for any $(A,Z)$ in the Clemens complex we have $\mathcal{F}_{A,Z} \cap \mathcal{F} = \{0\}$.
\end{prop}

\begin{proof}
    As before we write $\mathcal{A}$ for the set of irreducible components of $D$.  If $\alpha \in \mathcal{F}_{A,Z}$, then $\alpha$ has vanishing intersection against every divisor in $\mathcal{A} \setminus A$.  If furthermore $\alpha \in \mathcal{F}$, then $\alpha$ has vanishing intersection against every divisor in $\Sigma^{(1)} \setminus A$.  Since $A$ is a subset of the rays in a cone $\sigma$ of our fan, the divisors in $\Sigma^{(1)} \setminus A$ span $\Pic(X)_{\mathbb{Q}}$ and so there is a positive linear combination of these divisors that is big.  We conclude that $\alpha = 0$.
\end{proof}

Thus to count $\mathbb{A}^{1}$-curves we should sum up contributions from the faces $\mathcal{F}_{A,Z}$ as we vary $(A,Z)$ in the Clemens complex.  We next discuss the dimensions of the faces $\mathcal{F}_{A,Z}$ following \cite{Wilsch} and \cite{Santens}.  It is possible for a face $\mathcal{F}_{A,Z}$ to have pathological behavior in the sense that its dimension fails to be predicted by the Clemens complex.   For example, \cite{Wilsch} identifies a toric pair $(X,D)$ and a non-trivial element $(A,Z)$ in the Clemens complex such that $\mathcal{F}_{A,Z} = \mathcal{F}_{\emptyset,X}$.   
We can systematically address such pathologies using analytic obstructions in the sense of \cite[Definition 3.14]{Santens}.

\begin{defi}
    Let $(\underline{X},D)$ be a split SNC pair. Let $(A, Z)$ be an element of the Clemens complex of this pair.
    We define
    \[
    \underline{X}^\circ_Z = \underline{X} \setminus (\cup_{(A', Z') \not\leq (A, Z)}Z') = \underline{X}^\circ \cup \bigcup_{(A', Z') \leq (A, Z)}Z'^\circ.
    \]
    We say $(A, Z)$ admits an analytic obstruction if 
    \[
    \rH^0(\underline{X}^\circ_Z, \mathcal O_{\underline{X}^{\circ}_{Z}}) \neq \mathbf k.
    \]
\end{defi}

\begin{prop}
\label{prop:obstruction}
Let $(X,D)$ be a geometrically $\mathbb A^1$-connected smooth projective split toric pair. Let $(A, Z)$ be an element of the Clemens complex.  The following conditions are equivalent:
\begin{enumerate}
    \item There is no analytic obstruction for $(A, Z)$.
    \item The divisor $\sum_{i \in \mathcal{A} \backslash A} D_{i}$ has Iitaka dimension $0$.
    \item $\underline{X}^{\circ}_{Z}$ is rationally connected.
    \item There is a family of $\mathbb A^1$-curves passing through $2$ general points of $\underline{X}^\circ$ which map $\infty$ to $\underline{Z}^\circ$.
\end{enumerate}
\end{prop}

\begin{proof}
    We may assume that our ground field is algebraically closed.
    
    (1) $\implies$ (4): suppose that there is no analytic obstruction. 
    This means that if we let $\Sigma_Z$ be the fan associated to $\underline{X}^\circ_Z$, then $|\Sigma_Z|^\vee = 0$. In other words, $\mathrm{Cone}(\Sigma_Z) = N_{\mathbb R}$.
    We claim that there exists a contact order $\mathbf c_\infty$ in the sense of \cite[Section 8.1.2]{CLT25} which is positive for every $i \in A$ and $0$ for $i \in \mathcal{A} \backslash A$ such that
    \[
    -\mathbf c_\infty \in \mathrm{Cone}(v_i \,|\, i \not\in \mathcal A)^\circ.
    \]
   Suppose that this is not true. Then one can find a hyperplane which separates the two cones
   \[
   \mathrm{Cone}(v_i \,|\, i \not\in \mathcal A), \quad \mathrm{Cone}(-v_i \,|\, i \in A).
   \]
   However, this contradicts with $\mathrm{Cone}(\Sigma_Z) = N_{\mathbb R}$.
   Thus after choosing $\mathbf c_\infty$ general, we may write
   \[
   -\mathbf c_\infty = \sum_{i \not\in \mathcal A} c_iv_i,
   \]
   with $c_i \geq 0$ such that $\{v_i | c_i > 0\}$ spans $N_{\mathbb R}$.
   Hence our assertion follows from \cite[Theorem 8.2]{CLT25}.   

   (4) $\implies$ (3): obvious

   (3) $\implies$ (2): suppose that $\underline{X}^{\circ}_{Z}$ is rationally connected.  If $\sum_{i \in \mathcal{A} \backslash A} D_{i}$ had positive Iitaka dimension, it would be linearly equivalent to a divisor $L$ whose support contains a general point $x$ of $\underline{X}^{\circ}$.  Using the rationally connected condition, we can find a rational curve $\underline{f}: \underline{\mathbb{P}}^{1} \to \underline{X}_{Z}^{\circ}$ through $x$ that is not contained in $\Supp(L)$.  Such a curve must have positive intersection against $L$, contradicting the fact that $\underline{f}(\underline{\mathbb{P}}^{1}) \subset X \backslash (\cup_{i \in \mathcal A\setminus A} D_i)$.

   (2) $\implies$ (1): Since $\sum_{i \in \mathcal{A} \backslash A} D_{i}$ has Iitaka dimension $0$, it can be contracted by a birational contraction $\phi: X \dashrightarrow Y$ to a normal projective toric variety $Y$.  Since $Y$ and $\underline{X}^{\circ}_{Z}$ are isomorphic outside of codimension $\geq 2$ subsets, we have $H^{0}(\underline{X}^{\circ}_{Z},\mathcal{O}_{\underline{X}^{\circ}_{Z}}) = H^{0}(Y,\mathcal{O}_{Y}) = \mathbf k$.
\end{proof}

\begin{rema}
    Even when $(A,Z)$ admits an analytic obstruction, it is possible for there to be $\mathbb{A}^{1}$-curves meeting $\underline{X}^{\circ}$ whose closure meets $\underline{Z}^{\circ}$. 
    However, if $\sum_{i \in \mathcal{A} \backslash A} D_{i}$ has positive Iitaka dimension then all the $\mathbb{A}^{1}$-curves lying on the corresponding face $\mathcal{F}_{A,Z}$ will be contained in the fibers of a non-trivial toric rational map from $X$. 
    Over a finite field such curves will never be Zariski dense.  
\end{rema}

Furthermore, the following result shows that when there is no analytic obstruction the dimension of the face $\mathcal{F}_{A,Z}$ can be computed directly from the Clemens complex.

\begin{coro}
    Assume that $(A,Z)$ has no analytic obstruction. Then the dimension of $\mathcal{F}_{A,Z}$ is the sum of $\rho(\underline{X}^\circ)$ plus the dimension of $(A,Z)$ as an element of the Clemens complex.
  
\end{coro}

\begin{proof}
   By Proposition~\ref{prop:obstruction} (2), $D_i (i \in \mathcal A \setminus A)$ spans an extremal face of the effective cone of divisors such that each $D_i$ is an extremal ray. Since $\mathcal F_{A, Z}$ is the dual face of this face, we conclude
   \[
   \dim \mathcal F_{A, Z} = \rho(\underline{X}) - (\#\mathcal A- \#A).
   \]
   On the other hand, by Proposition~\ref{prop:equivconditionstoric}, we have
   $\rho(\underline{X}^\circ) = \rho(\underline{X}) - \#\mathcal A$. Thus our assertion follows.
\end{proof}

\begin{rema}
    The above corollary fails when $(A,Z)$ admits an analytic obstruction. See \cite{Wilsch} for such an example.
\end{rema}

\subsection{The virtual height zeta function for $\mathbb A^1$-curves}

Now we proceed as in Section~\ref{sec:Campana} and set up the virtual height zeta function for $\mathbb A^1$-curves.
 As before, first one should note the following lemma:

\begin{lemm}
\label{lemm:countingproductII}
    We have
    \[
    \#\underline{M}_{\alpha, D}^\circ(\mathbf k) = (q-1)^n \#\underline{U}_{\bold r, D}(\mathbf k).
    \]
\end{lemm}
\begin{proof}
    This follows from Lemma~\ref{lemm:isotriviality}.
\end{proof}
Thus it suffices to analyze $\#\underline{U}_{\bold r, D}(\mathbf k)$.
We consider the following indicator function $\delta_{D} : \sqcup_{\bold r}H_{\bold r}(\mathbf k) \to \{0, 1\}$:
\[
\delta_{D}(w):= 
\begin{cases}
    1 & \text{ if } w \in \sqcup_{\bold r}U_{\bold r, D} \\
    0 & \text{ otherwise.}
\end{cases}
\]
As before we have
\[
\#\underline{U}_{\bold r, D}(\bold k) = \sum_{w \in \underline{H}_{\bold r}(\mathbf k)} \delta_{D}(w).
\]
In view of Proposition~\ref{prop:numericalclasswithA1}, it is natural to count $\mathbb A^1$-curves associated to an element of the Clemens complex. This is very much in the spirit of \cite[Section 6]{Santens}.

In the rest of this section, we fix an element $(A, Z)$ of the Clemens complex. 
Note that we have $Z = \underline{\Delta}_A$ where as before $\underline{\Delta}_{A} = \cap_{i \in A} \underline{\Delta}_{i}$. 
We write $(t_i)_{i \in (\Sigma^{(1)}\setminus \mathcal A) \cup A}$ as $\bold t_A$.
We consider the following virtual height zeta function:
\begin{align*}
\mathsf Z_{A, Z}(\bold t_A) &= \sum_{\text{$\bold r$ satisfies $\diamond_A$}} \left(\prod_{i \in (\Sigma^{(1)}\setminus \mathcal A) \cup A}q^{-t_ir_i}\right)\#\underline{U}_{\bold r, D}(\bold k)\\ &= \sum_{\text{$\bold r$ satisfies $\diamond_A$}}\sum_{w \in \underline{H}_{\bold r}(\mathbf k)} \left(\prod_{i\in (\Sigma^{(1)}\setminus \mathcal A) \cup A}q^{-t_ir_i}\right)  \delta_{D}(w),
\end{align*}
where $\diamond_A$ indicates
\[
r_i = 0 \text{ for any $i \in \mathcal A\setminus A$.}
\]
The above zeta function can be written as an Euler product: 
\[
\mathsf Z_{A, Z}(\bold t_A) = \prod_{c \in |\underline{\mathbb A}^1|}\left(\sum_{\text{$\bold r$ satisfies $\diamond_A$}}\prod_iq_c^{-t_ir_i}\delta_{D}(\bold r)\right) \times \left(\sum_{\text{$\bold r$ satisfies $\diamond_A$}}\prod_iq_\infty^{-t_ir_i}\delta_{\infty}(\bold r)\right).
\]
Let $\underline{T}(\mathbb A_{F})_{D, A}$ be
\[
\underline{T}(\mathbb A_F) \cap \left(\prod_{c \in |\underline{\mathbb{A}}^1|} X^\circ(\mathfrak o_c) \times V(\mathfrak o_\infty) \right),
\]
where $V = X \setminus (\cup_{i \in \mathcal A \setminus A}\Delta_i)$. We denote its indicator function by $\delta_{D, A} : \underline{T}(\mathbb A_F) \to \{0, 1\}$.
We consider the height integral
\[
\mathcal I(\delta_{D, A}; \bold t_A) = \int_{\underline{T}(\mathbb A_F)} \mathsf H(\bold t_A - \bold 1, (g_c))^{-1}\delta_{D, A}((g_c)) \mathrm d\tau_{\underline{T}},
\]
where we assume that $\bold t$ satisfies $\diamond_A$.
This becomes the Euler product
\[
\mathcal I(\delta_{D, A};\bold t_A) = L_*(1, \mathrm{EP}(\underline{T}))^{-1}\prod_{c \in |\underline{\mathbb P}^1|} \mathcal I_{D, A, c}(\bold t_A),
\]
where $\mathcal I_{D, A, c}(\bold t_A)$ is defined as
\[
\mathcal I_{D, A, c}( \bold t_A) = 
\begin{cases}
\int_{\underline{X}^\circ(\mathfrak o_c)} \mathsf H_c(\bold t_A - \bold 1, g_c)^{-1} L_c(1, \mathrm{EP}(\underline{T})) \mathrm d\tau_c & \text{ if $c \in |\underline{\mathbb A}^1|$}\\
\int_{\underline{V}(\mathfrak o_\infty)} \mathsf H_c(\bold t_A - \bold 1, g_c)^{-1} L_c(1, \mathrm{EP}(\underline{T})) \mathrm d\tau_c & \text{ if $c = \infty$}.
\end{cases}
\]
By Denef's formula (\cite[Proposition 4.5]{CLT10}), when $c \in |\mathbb A^1|$ we have
\[
\mathcal I_{D, A, c} (\bold t_A) = \left(\frac{q_c}{q_c-1}\right)^n\sum_{B \subset (\Sigma^{(1)} \setminus \mathcal A)} q_c^{-n} \#\underline{\Delta}_B^\circ(\mathbf k_c) \prod_{i \in B}(q_c-1)\frac{q_c^{-t_i}}{1-q_c^{-t_i}}.
\]
When $c = \infty$, we have
\[
\mathcal I_{D, A, c} (\bold t_A) = \left(\frac{q_c}{q_c-1}\right)^n\sum_{B \subset (\Sigma^{(1)}\setminus\mathcal A)\cup A} q_c^{-n} \#\underline{\Delta}_B^\circ(\mathbf k_c) \prod_{i \in B}(q_c-1)\frac{q_c^{-t_i}}{1-q_c^{-t_i}}.
\]
Again since $\underline{\Delta}_B^\circ$ is a split torus, these local integrals become
\[
\mathcal I_{D, A, c} ( \bold t_A) = 
\begin{cases}\sum_{B \subset (\Sigma^{(1)}\setminus \mathcal A)}  \delta'(B) \prod_{i \in B}\left(\sum_{r_i = 1}^{\infty}q_c^{-r_it_i}\right) & \text{ if $c \in |\mathbb A^1|$;}\\
\sum_{B \subset (\Sigma^{(1)}\setminus\mathcal A)\cup A}  \delta'(B) \prod_{i \in B}\left(\sum_{r_i = 1}^{\infty}q_c^{-r_it_i}\right) & \text{ if $c = \infty$.}
\end{cases}
\]
where as before
$\delta'(B)$ is $1$ if $\underline{\Delta}_B\neq \emptyset$ and $0$ otherwise.
Then we have
\[
\mathcal I_c(\delta_{D, A, c}; \bold t_A) = 
\begin{cases}
\sum_{\text{$\bold r$ satisfies $\diamond_A$}}\prod_iq_c^{-t_ir_i}\delta_{D}(\bold r) & \text{ if $c \in |\mathbb A^1|$;}\\
\sum_{\text{$\bold r$ satisfies $\diamond_A$}}\prod_iq_c^{-t_ir_i}\delta_{\infty}(\bold r) & \text{ if $c = \infty$.}
\end{cases}
\]
Using this we obtain
\begin{prop}
\label{prop:A1heightintegral}
    There exists a positive constant $\epsilon > 0$ such that the function
    \[
    \prod_{i \in \Sigma^{(1)}\setminus \mathcal A}(1-q^{-(t_i -1)}) \prod_{i \in A}(1-q^{-t_i})\mathsf Z_{A, Z}(\bold t_A)
    \]
    admits a holomorphic continuation to the domain $\mathsf T_{> -\epsilon}$ defined by 
    $$
    \begin{cases}
    \Re(t_i)\geq 1 - \epsilon & \text{ if $i \in \Sigma^{(1)}\setminus \mathcal A$}\\
    \Re(t_i) \geq -\epsilon & \text{ if $i \in A$}.
    \end{cases}
    $$ and moreover it satisfies
    \begin{align*}
    &\lim_{\star}\prod_{i \in \Sigma^{(1)}\setminus \mathcal A}(1-q^{-(t_i-1)})\prod_{i \in A}(1-q^{-t_i})\mathsf Z_{A, Z}(\bold t_A) \\
&= \left(\frac{q}{q-1}\right)^{n} \int_{\underline{X}^\circ(\mathfrak o^{\infty})}1 \, \mathrm d\tau_{\underline{X}^\circ}^\infty \\
&\qquad\qquad \times (\log q)^{\# A}\int_{\underline{\Delta}_A^{\circ'} (\mathfrak o_\infty)} \mathsf H_\infty(-K_X, z_\infty)^{-1}\mathrm L_\infty(1, \mathrm{EP}(\underline{X}^\circ))\mathrm d\tau_{\underline{\Delta}_A, \infty}
    \end{align*}
    where $\star$ indicates the following limit:
    \[
    \begin{cases}
        t_i \to 1 & \text{ if $i \in \Sigma^{(1)} \setminus \mathcal A$}\\
        t_i \to 0 & \text{ if $i \in A$},
    \end{cases}
    \]
    $\underline{\Delta}_A^{\circ'}$ is the strata with respect to $D$, and
    \[
    \underline{X}^\circ(\mathfrak o^{\infty}) = \prod_{c \in |\underline{\mathbb A}^1|} \underline{X}^\circ(\mathfrak o_c), \quad \tau_{\underline{X}^\circ}^\infty = L_*(1, \mathrm{EP}(\underline{X}^\circ))^{-1}\prod_{c \in |\underline{\mathbb A}^1|} L_c(1, \mathrm{EP}(\underline{X}^\circ))\tau_c.
    \]
    We denote this leading constant as $\left(\frac{q}{q-1}\right)^{n}\tau((X, D), A)$.
\end{prop}

\begin{proof}
    The first statement follows from the proof of \cite[Lemma 4.1]{CLT10}.
    The second statement follows from the proofs of \cite[Proposition 4.10]{CLT10} and \cite[Proposition 4.3]{CLT10}.
\end{proof}

Thanks to this result we have the following counting estimates:
\begin{prop}
\label{prop:countingA1}
    There exists $\eta > 0$ such that assuming $\bold r$ satisfies $\diamond_A$, we have
    \[
    \frac{\#\underline{U}_{\bold r, D}(\bold k)}{q^{\sum_{i \in \Sigma^{(1)} \setminus \mathcal A} r_i}} = \frac{q^n}{(q-1)^n}\tau((X, D), A) + O(q^{-\eta\min_{i \in (\Sigma^{(1)} \setminus \mathcal A) \cup A}\{r_i\}}).
    \]
\end{prop}

\begin{proof}
    One may argue as in Proposition~\ref{prop:counting}.
\end{proof}

\subsection{The main result}

In the view of Proposition~\ref{prop:numericalclasswithA1}, we consider the counting function associated to a triple $(X, D, A)$. To this end, we introduce the following definitions:
\begin{defi}
    Let $N_1((X, D), A) \subset N_1(\underline{X})$ be the vector space spanned by the face $\mathcal F_{A, Z}$. We denote $N_1((X, D), A)\cap N_1(\underline{X})_{\mathbb Z}$ by $N_1((X, D), A)_{\mathbb Z}$.
    We define the index $\mathsf r((X, D), A)$ by
    \[
    \mathsf r((X, D), A) = \min\{ -K_X. \alpha > 0 \, | \, \alpha \in N_1((X, D), A)_{\mathbb Z} \}.
    \]
\end{defi}

Now we define the counting function we are interested in:
for any positive integer $d$, we define
\[
N(((X, D), A), \mathsf r((X, D), A)d)  = \sum_{\alpha \in \mathcal F_{A, Z, \mathbb Z}, -K_X.\alpha \leq \mathsf r((X, D), A)d} \#M_{\alpha, D}^\circ(\bold k).
\]
This is well-defined due to the Northcott property established in Proposition~\ref{prop:northcott}.
This counting function is analogous to the set up considered in \cite[Section 6]{Santens}.
Now we state our main theorem:
\begin{theo}
\label{theo:A1curve}
 Assume that $(A,Z)$ has no analytic obstruction.
    Then we have
    \[
    N(((X, D), A), \mathsf r((X, D), A)d) \sim \mathsf c((X, D), A)q^{\mathsf r((X, D), A)d}(\mathsf r((X, D), A)d)^{b-1},
    \]
    as $d \to \infty$ where $\mathsf c((X, D), A)$ is given by
    \[
     \mathsf c((X, D), A) = \frac{\alpha(\mathcal F_{A, Z})q^n}{(1-q^{-\mathsf r((X, D), A)})}\tau((X, D), A),
    \]
    and $b = \dim N_1((X, D), A) = \rk \, \Pic(\underline{X}^\circ) + \dim A$.
    Here $\alpha(\mathcal F_{A, Z})$ is defined using the volume with respect to the vector space $V_A = \langle \mathcal F_{A, Z} \rangle$ and the lattice $V_{A, \mathbb Z} = V_A\cap N_1(\underline{X})_{\mathbb Z}$.

\end{theo}

\begin{proof}
This follows from Proposition~\ref{prop:countingA1} using the counting arguments of Theorem~\ref{theo:Campana} involving shrunken cones. We briefly sketch the proof.

As before, we consider the shrunken cone:
\[
\mathcal{F}_{A,Z, \epsilon} = \{\alpha \in \mathcal{F}_{A,Z} \, | \, \Delta_i.\alpha \geq -\epsilon K_X. \alpha \text{ for any $i \in (\Sigma^{(1)}\setminus \mathcal A)\cup A$}\},
\]
and we denote the closure of its complement:
\[
\mathsf C_\epsilon = \overline{\mathcal{F}_{A,Z} \setminus \mathcal{F}_{A,Z, \epsilon}}.
\]
Using Proposition~\ref{prop:countingA1}, since there is no analytic obstruction, we can prove that
\[
N(\mathcal{F}_{A,Z, \epsilon}, \mathsf r((X, D), A)d) \sim \mathsf c(\mathcal{F}_{A,Z, \epsilon})q^{\mathsf r((X, D), A)d}(\mathsf r((X, D), A)d)^{b-1},
\]
as $d \to \infty$ where the constant $\mathsf c(\mathcal{F}_{A,Z, \epsilon})$ is given by
\[
 \mathsf c(\mathcal{F}_{A,Z, \epsilon}) = \frac{\alpha(\mathcal{F}_{A,Z, \epsilon})q^n}{(1-q^{-\mathsf r((X, D), A)})}\tau((X, D), A).
\]
Now we have inequalities:
 \begin{align*}
    &\frac{N(\mathcal{F}_{A,Z, \epsilon}, \mathsf r((X, D), A)d)}{q^{\mathsf r((X, D), A)d}(\mathsf r((X, D), A)d)^{b-1}}\leq \frac{N(((X, D), A), \mathsf r((X, D), A)d)}{q^{\mathsf r((X, D), A)d}(\mathsf r((X, D), A)d)^{b-1}}\\
    &\leq
    \frac{N(\mathcal{F}_{A,Z, \epsilon}, \mathsf r((X, D), A)d)}{q^{\mathsf r((X, D), A)d}(\mathsf r((X, D), A)d)^{b-1}} + \frac{N(\mathsf C_\epsilon, \mathsf r((X, D), A)d)}{q^{\mathsf r((X, D), A)d}(\mathsf r((X, D), A)d)^{b-1}}.
    \end{align*}
    Using Proposition~\ref{prop:countingA1}, we can also prove that
    \[
    \limsup_{d \to \infty} \frac{N(\mathsf C_\epsilon, \mathsf r((X, D), A)d)}{q^{\mathsf r((X, D), A)d}(\mathsf r((X, D), A)d)^{b-1}} \leq \frac{(q-1)^n\mathsf C}{(1-q^{-\mathsf r((X, D), A)})}\alpha(\mathsf C_\epsilon).
    \]
    Thus we obtain
    \begin{align*}
    &\mathsf c(\mathcal{F}_{A,Z, \epsilon})\leq \liminf_{d \to \infty}\frac{N(((X, D), A), \mathsf r((X, D), A)d)}{q^{\mathsf r((X, D), A)d}(\mathsf r((X, D), A)d)^{b-1}}\\
    &\leq \limsup_{d \to \infty}\frac{N(((X, D), A), \mathsf r((X, D), A)d)}{q^{\mathsf r((X, D), A)d}(\mathsf r((X, D), A)d)^{b-1}}\leq\mathsf c(\mathcal{F}_{A,Z, \epsilon}) + \frac{(q-1)^n\mathsf C}{(1-q^{-\mathsf r((X, D), A)})}\alpha(\mathsf C_\epsilon).  
    \end{align*}
    As $\epsilon \to 0$, our assertion follows.
\end{proof}

\begin{rema}
    Our results suggest that the strong approximation holds for $((X, D); A)$ off the infinity in the sense of \cite{Santens} when $(A,Z)$ has no analytic obstruction.
\end{rema}

\begin{rema}
Our counting method does not apply when $(A,Z)$ has an analytic obstruction. Indeed, in such a situation, \cite[Theorem 3.12]{Santens} shows that the set of $\mathbb A^1$-curves cannot be Zariski dense. There is always a face without analytic obstruction, however, it is not clear to us that the main term of the counting function of all $\mathbb A^1$-curves is formed by faces without analytic obstruction. In the view of \cite[Theorem 3.12]{Santens}, we think that we should include those $\mathbb A^1$-curves associated to faces with analytic obstructions to the exceptional set.
\end{rema}

\section{Appendix: Toric varieties and their Brauer groups over global function fields}
\label{sec:appendix}

Our goal of this section is to prove Lemma~\ref{lemm:Loughran}. We closely follow the exposition of \cite[Section 4]{Loughran} and use the notation established there.
 Let $\mathbf k = \mathbb F_q$ be a finite field and $\underline{C}$ be a smooth geometrically integral projective curve defined over $\mathbf k$. Let $F = \mathbf k(\underline{C})$ be a global function field and $\underline{T}$ be an algebraic torus defined over $F$. We denote the absolute Galois group by $G_F = \mathrm{Gal}(F^s/F)$, and we denote the following Galois module
\[
\mathrm{Hom}(\underline{T}_{F^s}, \underline{\mathbb G}_m)
\]
by $X^*(\underline{T}_{F^s})$. One should note that
\[
X^*(\underline{T}) := X^*(\underline{T}_{F^s})^{G_{F}}
\]
is the group of characters of $\underline{T}$.
We denote the groups of cocharacters by $X_*(\underline{T}_{F^s})$ and $X_*(\underline{T})$. We have the dual relation:
\[
X_*(\underline{T}_{F^s}) = \mathrm{Hom}(X^*(\underline{T}_{F^s}), \mathbb Z).
\]
The splitting field of $\underline{T}$ is the fixed field of the kernel of $$G_F \to \mathrm{GL}(X^*(\underline{T}_{F^s})).$$ Over this field, the base change of $\underline{T}$ is isomorphic to $\underline{\mathbb G}_m^n$.

For $c \in |\underline{C}|$, let the group $\underline{T}(\mathfrak o_c)$ be the maximal compact subgroup of $\underline{T}(F_c)$ and define the pairing:
\[
\underline{T}(F_c) \times X^*(\underline{T}_{F_c}) \to \mathbb Z, (g_c, \chi) \mapsto \frac{\log |\chi(g_c)|}{\log q_c},
\]
which induces the exact sequence:
\[
0 \to\underline{T}(\mathfrak o_c) \to \underline{T}(F_c) \to \mathrm{Hom}(X^*(T_{F_c}), \mathbb Z).
\]
The third homomorphism is surjective if $c$ is unramified in the splitting field of $\underline{T}$.
This local pairing gives rise to the adelic pairing:
\[
\underline{T}(\mathbb A_F) \times X^*(T) \to \mathbb Z, ((g_c), \chi) \mapsto \sum_{c \in |\underline{C}|}\frac{\log |\chi(g_c)|}{\log q}.
\]
We denote the left kernel of this pairing by $\underline{T}(\mathbb A_F)^1$ and we obtain the exact sequence
\[
0 \to \underline{T}(\mathbb A_F)^1 \to \underline{T}(\mathbb A_F) \to \mathrm{Hom}(X^*(\underline{T}), \mathbb Z).
\]

Let 
\[
\Sha(\underline{T}) = \ker \left(\rH^1(F, \underline{T}_{F^s}) \to \prod_{c \in |\underline{C}|} \rH^1(F_c, \underline{T}_{F^s_c})\right)
\]
be the Tate--Shafarevich group of $\underline{T}$. This is finite over global fields by \cite{conrad}.
We also let
\[
\mathfrak B(\underline{T}) = \ker\left(\Br_1(\underline{T}) \to \prod_{c \in |\underline{C}|} \Br_1(\underline{T}_{F_c})\right).
\]
Using \cite[Lemma 6.2 and Lemma 6.8]{Sansuc}, we obtain the canonical pairing:
\[
\Sha(\underline{T}) \times \mathfrak B(\underline{T}) \to \mathbb Q/\mathbb Z,
\]
which induces the canonical homomorphism
\[
\Sha(\underline{T})  \to \mathfrak B(\underline{T})^{\sim},
\]
where $\mathfrak B(\underline{T})^{\sim} = \mathrm{Hom}(\mathfrak B(\underline{T}), \mathbb Q/\mathbb Z)$.
Let $\ell$ be a prime not equal to the characteristic of $\mathbf k$.
The proof of \cite[Proposition 8.3]{Sansuc} shows that the above homomorphism induces an isomorphism
\[
\Sha(\underline{T})\{\ell\}  \cong \mathfrak B(\underline{T})^{\sim}\{\ell\},
\]
where $M\{\ell\}$ means the $\ell$-primary part for an abelian group $M$.
Note that \cite[Proposition 8.3]{Sansuc} is stated over number fields, but the proof works over function fields for the $\ell$-primary part without any modification. A key to this is Tate--Nakayama duality theory which is also valid over global function fields as long as we are concerned about the $\ell$-primary part. See \cite[Chapter II Proposition 4.14]{Milne}.

Let 
\[
\Br_e(\underline{T}) = \{ b \in \Br_1(\underline{T})\, |\, b(1) = 0 \},
\]
where $\Br_1(\underline{T})$ is the kernel of the pullback map to $\underline{T}_{F^s}$.
Regarding this group, we have the following lemmas:
\begin{lemm}[{\cite[Lemma 6.9(ii)]{Sansuc}}]
    There are natural isomorphisms
    \[
    \Pic(\underline{T}) \cong \rH^1(F, X^*(\underline{T}_{F^s})), \quad \Br_e(\underline{T}) \cong \rH^2(F, X^*(\underline{T}_{F^s})).
    \]
\end{lemm}

\begin{lemm}[{\cite[Lemma 6.9]{Sansuc}}]
    The pairing 
    \[
    \Br_e(\underline{T}) \times \underline{T}(F) \to \Br(F), \, (b, g) \mapsto b(g)
    \]
    is bilinear.
\end{lemm}

Note that this pairing admits the following another description:
\begin{equation}
\label{equation:pairing}
\xymatrix{
\Br_e(\underline{T})\ar[d]^*[@]{\cong} & \times & \underline{T}(F) \ar[d]^*[@]{\cong}\ar[r]& \Br(F)\ar@{=}[d]\\
\rH^2(F, X^*(\underline{T}_{F^s})) & \times  & \rH^0(F, \underline{T}(F^s)) \ar[r]& H^2(F, F^{s \times}),
}
\end{equation}
where the bottom pairing is the cup product.

Next we state the following theorem:
\begin{theo}[{\cite[Theorem 4.4]{Loughran}}]
    For any $c \in |\underline{C}|$, the bilinear pairing
    \[
    \Br_e(\underline{T}_{F_c}) \times \underline{T}(F_c) \to \Br(F_c) \cong \mathbb Q/\mathbb Z
    \]
    induces an isomorphism
    \[
    \Br_e(\underline{T}_{F_c}) \cong \underline{T}(F_c)^\sim
    \]
    of abelian groups.
\end{theo}
\begin{proof}
    This follows from \cite[Chapter I Corollary 2.4]{Milne}, (\ref{equation:pairing}), and \cite[Lemma 4.3]{Loughran}.
\end{proof}

Finally we prove the following theorem:
\begin{theo}[{\cite[Theorem 4.5]{Loughran}}]
\label{theo:duality}
    The pairing
    \[
    \Br_e(\underline{T}) \times \underline{T}(\mathbb A_F)/\underline{T}(F) \to \mathbb Q/\mathbb Z
    \]
    is bilinear and it induces the exact sequence
    \[
    0 \to \mathfrak B(T) \to \Br_e(\underline{T}) \to (\underline{T}(\mathbb A_F)/\underline{T}(F) )^\sim \to 0.
    \]
\end{theo}
\begin{proof}
    This follows from the proof of \cite[Theorem 4.5]{Loughran}. In its notation, Loughran claimed that $T(\mathbb A_F)/T(F)$ is a closed subgroup of finite index in $G(T)$.
    This finite index is measured by $\Sha(\underline{T})$ which is also finite over global function fields by \cite{conrad}.
    Thus the proof of \cite[Theorem 4.5]{Loughran} goes through.
\end{proof}

\begin{coro}[{\cite[Corollary 4.6]{Loughran}}]
    Suppose that $T$ is rational. Then $\mathfrak B(T) = 0$ so that we have an isomorphism
    \[
    \Br_e(\underline{T}) \cong (\underline{T}(\mathbb A_F)/\underline{T}(F) )^\sim.
    \]
\end{coro}
\begin{proof}
    See the proof of \cite[Corollary 4.6]{Loughran}.
\end{proof}

\begin{proof}[Proof of Lemma~\ref{lemm:Loughran}]
    This follows from Theorem~\ref{theo:duality} and the second commutative diagram in \cite[Lemma 4.7]{Loughran}. This second diagram is valid in our setting. A key to this is the fact that $\Delta_i$ is rational for any $i$. See \cite[Lemma 3.25]{SS24} for more details.
\end{proof}

\bibliographystyle{alpha}
\bibliography{CampanaCurve}

\end{document}